\documentclass[8pt]{article}

\usepackage[utf8]{inputenc}
\usepackage[lined]{algorithm2e}
\usepackage{minitoc}
\usepackage{mathtools}

\usepackage{stmaryrd}

\usepackage{fancyvrb}
\usepackage{color}
\usepackage{multicol}

\usepackage{datetime}

\usepackage[left=3cm,right=3cm]{geometry}
\usepackage{soul}
\usepackage{cancel}

\usepackage[pdfauthor={Andres A Leon Baldelli},pdftitle={},pdfsubject={},pdfkeywords={},pdfproducer={Latex with hyperref},pdfcreator={pdflatex}]{hyperref}

\usepackage[english]{babel}

\usepackage{booktabs}

\usepackage{mathrsfs}

\usepackage{amsmath, amsthm, amssymb}

\theoremstyle{plain} 
\newtheorem{theorem}{Theorem}[section] 
\newtheorem{proposition}[theorem]{Proposition} 
\newtheorem{lemma}[theorem]{Lemma}

\theoremstyle{definition} 

\theoremstyle{remark} 
\newtheorem{remark}[theorem]{Remark}

\renewenvironment{proof}[1][Proof]{\noindent\textbf{#1.} }{\hfill$\square$\vspace{5pt}}

\usepackage{latexsym}

\usepackage{boxedminipage}

\usepackage{listings}

\usepackage{minitoc}

\usepackage{ifpdf}

\usepackage{subcaption}

\usepackage{array}

\usepackage[utf8]{inputenc}   
\usepackage[T1]{fontenc}      
\usepackage{amsmath, amssymb} 
\usepackage{graphicx}
\graphicspath{ {.}, {./images} }
\usepackage{xcolor}           
\usepackage{hyperref}         
\usepackage[draft]{changes}
\setdeletedmarkup{}
\definechangesauthor[color=orange]{ALB}

\usepackage{booktabs}
\usepackage{geometry}         
\usepackage{float}            
\usepackage{enumitem}         
\usepackage{caption}

\newcommand{\edgin}{\mathcal{E}_{\text{in}}}
\newcommand{\edgbd}{\mathcal{E}_{\text{bd}}}
\newcommand{\cof}{\operatorname{cof}}

\newcommand{\wexact}{w^*}
\newcommand{\vexact}{v^*}

\renewcommand{\O}{\Omega}
\newcommand{\uv}{v}
\newcommand{\uw}{w}

\title{Kinematically incompatible F\"oppl-von K\'arm\'an  plates: analysis and numerics   }
\author{Edoardo Fabbrini%
\footnote{Graduate School of Mathematics, Kyushu University, Japan 
 \url{fabbrini.edoardo.840@s.kyushu-u.ac.jp}}, 
Andrés A. León Baldelli%
\footnote{
Sorbonne Université, CNRS,
Institut Jean Le Rond d’Alembert, \url{leon.baldelli@cnrs.fr}
}, 
Pierluigi Cesana%
\footnote{Institute of Mathematics for Industry, Kyushu University, Japan \url{cesana@math.kyushu-u.ac.jp}
}}

\begin{document}

\ifpdf
	\DeclareGraphicsExtensions{.pdf, .jpg, .tif}
\else
	\DeclareGraphicsExtensions{.eps, .jpg}
\fi

\maketitle


\begin{abstract}
	We investigate thin plates where out-of-plane deformations arise due to membrane kinematic incompatibility of rotational type, specifically Volterra wedge disclinations, which are commonly observed in metal plates and graphene.
	We present theorems that guarantee the existence and regularity of equilibrium solutions in the presence of a finite number of disclinations and a dead load, for clamped plates.
	To solve the equilibrium equations, we implement a numerical code in the FEniCS environment and apply it to a series of parametric test studies. Our Finite Element method follows the Discontinuous Galerkin approach with
	$C^0$  elements.
\end{abstract}


\tableofcontents

\section{Introduction}

The two-dimensional F\"oppl-von K\'arm\'an equations (FvK), originally developed by L. F\"oppl \cite{foppl} and T.  von K\'arm\'an
\cite{vonkarman}, provide a modeling framework for describing the behavior of thin, nonlinear elastic plates subjected to bending and stretching \cite{ciarlet97}, \cite{ciarlet80}.
They have long been considered the classical model in structural mechanics  to describe  moderately large deflections  in thin shells and plates, where the deflection is large compared to the plate's thickness.
The FvK model couples the large transverse (out-of-plane) displacement with the in-plane membrane (stretching) displacements, leading to a nonlinear relationship between strain and displacement components. This coupling allows for the modeling of geometric nonlinearity, including   moderate  deformations and rotations. Mathematically, the model is described by two fourth-order partial differential equations with two unknown scalar fields: the out-of-plane displacement (deflection) and the Airy stress function (or potential) of the system. The stress components are linearly related to the strain components and can be reconstructed by taking the second-order derivatives of the Airy potential.

The objective of this paper is to present a platform of both analytical results and numerical calculations for von K\'arm\'an  plates in relevant configurations, where out-of-plane displacements are induced by kinematic incompatibilities at the lattice level of the material, specifically in the form of wedge disclinations.

Wedge disclinations, defined by the corresponding Frank angle, represent classical lattice asymmetries first introduced by Volterra  in his celebrated paper \cite{V07}. Alongside dislocations,
disclinations play a key role in the plastic behavior     of metal alloys \cite{Orowan1934}, \cite{Polanyi1934}, \cite{ET67}, \cite{TE68}  and elastic crystals
\cite{CPL14}, \cite{I19}, \cite{IHM13}. They are
also  observed in graphene \cite{graphene}, \cite{Banhart11}, \cite{ARIZA10}.
The effect of these disclinations is to distort the crystal lattice by creating rotational asymmetries, leading to local curvature of the graphene sheet. This   can manifest as positive Gaussian curvature (caused by a positive Frank angle disclination, resulting in a small tent-like deformation) or negative Gaussian curvature (caused by a negative  disclination, leading to a small saddle-like deformation) \cite{liu10}.

Our main results are as follows. From the analytical standpoint, we present  (Section \ref{2501041121}) theorems that guarantee the existence and regularity of solutions to Dirichlet-type boundary value problems for the FvK equations in the presence of both dead loads (defined by a smooth forcing term for the deflection equation)  and kinematic incompatibilities (forced by a finite collection of Dirac deltas for the Airy potential equation).
Our results exploit the variational structure of the FvK equations and are derived by extending and generalizing the classical analysis of Ciarlet \cite{ciarlet97}, \cite{ciarlet80} for the case of perfect kinematic compatibility to incorporate membrane inelastic strains.
Unlike dislocations, which are topological defects accompanied by singular mechanical energy, disclinations are regarded as simple rotational asymmetries. As such, they have finite elastic energy, though they give rise to non-trivial mechanical configurations and are characterized by singular stresses.
Consequently, the mathematical analysis of the FvK equations in the presence of disclinations is based on the fact that the corresponding energy functionals and equations are non-singular, leading to well-posed, though non-convex, minimization problems in the calculus of variations.

From a numerical standpoint, we present a Finite Element formulation in Section \ref{sec:202501011737} and provide examples of numerical solutions in Section \ref{2501041129} for these boundary value problems.
Our work builds upon the codes developed in \cite{brunetti}, which we extend by incorporating point sources at the level of the Airy stress function equation to represent the presence of wedge disclinations.
Since we are dealing with fourth-order scalar elliptic equations, the choice of Finite Element class is a delicate issue.
We adopt the $C^0$ Discontinuous Galerkin interior penalty method developed in \cite{brunetti} for kinematically compatible (i.e., without disclinations) FvK plates and we extend it to incorporate the disclinations.
Section $\ref{sec:IPCDformulation}$ provides a concise and self-contained excursus on the penalty method for $C^0$   discontinuous elements. A comparison with other available approaches is presented in Section $\ref{2501041124}$.
The method we propose is fully variational, in the sense that it is based on a system of field equations and penalty/boundary integral terms that can be derived directly as the Euler-Lagrange equations of   a suitable functional.
The main advantage of our formulation lies in its flexibility: it can be easily generalized to incorporate additional loading and coupling terms, as long as they have a variational formulation, by simply adding them to the system functional.
Although in this contribution we focus on Dirichlet problems for both the Airy potential and deflection, other types of boundary conditions can be modeled by suitably adapting our codes.

The numerical modeling and mathematical analysis developed in this paper can be applied to model general systems, including both metallic and elastic-crystal plates with rotational kinematic incompatibilities.
In Section \ref{2501041148}, we present parametric simulations aimed at illustrating the properties of the model in dependence of relevant parameters. We focus on  configurations relevant to graphene in Section \ref{2501041149}.

\subsection{Literature review and comparison with existing work}

We refer to the work of Nelson's group, see \cite{SN88}, \cite{nelson} for the modeling of lattice defects in FvK plates, as occurring in flexible crystalline materials.
More recently, a general continuum theory of defects for FvK surfaces has been proposed in \cite{singh22}, along with the  construction of   special configurations involving disclinations and dislocations (see also \cite{D2SM00126H} and \cite{Pandey22}).
In our work, we provide analytical tools to rigorously justify the existence of such configurations in general, smooth enough, domains, including a precise characterization of the regularity properties of these solutions.

The particular case of isolated positive disclinations has been studied both analytically and numerically in \cite{singh21}, as well as in more general contexts within non-linear elasticity, including 3D non-linear elasticity, fully non-linear plates, and FvK plate models in \cite{Olbermann17} (see also \cite{Kupferman15}).

The mathematical investigation of disclinations and other lattice misfits, as well as topological defects, has captured the attention of mechanical modelers and analysts. See \cite{W68} for linear theories and \cite{Z97} for non-linear theories for disclinations. Mathematical analysis of models for finite systems of disclinations has been performed in \cite{CDLM24}. 
Other recent modeling approaches, based on the concept of  g.disclinations, a general method to treat phase transformations, grain boundaries, and other plasticity mechanisms, are presented  in \cite{Acharya15}, \cite{ZA18}, and \cite{Zhang18}. For an approach in nonlinear mechanics using differential geometry, we refer to \cite{Yavari13}. Additional modeling insights and analytical results on dislocations can be found in  \cite{DeLucaGarroniPonsiglione12}.

For more numerical simulations of graphene configurations  aimed at studying their mechanical properties in the presence of disclinations, based on an adaptation of the FvK equations, we refer to \cite{gao14} and references cited therein.
For more sophisticated models aimed at bridging the gap between continuum models and the quantum scale, including atomic-to-continuum limits,
we refer to
\cite{DAVINI17},
\cite{Davini18},
\cite{DAVINI19},
\cite{kotani}, and
\cite{Chen11}.

While we model plates as 2D structures with 3D deformations according to the FvK plate theory, we emphasize that a significant body of work has appeared over the years on the modeling of von K\'arm\'an plates in the presence of incompatible or inelastic strains, starting from 3D elasticity and applying 2D asymptotics in the regime of small thickness, see
\cite{Friesecke},
\cite{Bhattacharya16},
\cite{Lewicka11},
\cite{Lewicka14}.
There is also important work on modeling pre-strains simulating growth \cite{DeS19}, swelling \cite{DeS20}, and other related phenomena \cite{DeS22}, which all necessitate formulating the problem starting from the 3D model.
Research on the 3D-to-2D reduction for active materials, where shape changes and actuation are driven by the interplay of micro-scale strains induced by functional units coupled with large-scale strains, is presented in \cite{kruzik20}, \cite{DeS02}, \cite{DeS20b}, \cite{DeS17}.


\subsection{Notation}
We adopt standard notation for Lebesgue and Sobolev spaces (\cite{ciarlet97}). For any $\Omega \subset \mathbb{R}^2$, $p \ge 1$ and $k \in \mathbb{N}$ we use $W^{k, p}(\Omega) := \{f \in L^p(\Omega) \text{ s.t. } \tfrac{\partial^{|a|} f}{\partial x_1^{a_1} \partial x_2^{a_2}} \in L^p(\Omega) $ for every multi-index $ a $ with $ |a|\le k \}$. For $p=2$  we adopt the special notation $W^{k, 2}(\Omega) = H^k(\Omega)$. We use $H_0^2(\Omega)$ to denote the space of $H^2(\Omega)$ with vanishing trace, i.e. $H_0^2(\Omega) := \{f \in H^2(\Omega) \text{ s.t. } f = \partial_n f = 0 \text{ on } \partial \Omega \}$ (the boundary values of $f$ and $\partial_n f$ are to be intended in the sense of the trace operator where $n$ is the outer unit normal to $\partial \Omega$). $H^{-2}(\Omega)$ denotes the topological dual space of $H^2_0(\Omega)$ and $\langle \cdot, \cdot \rangle$ the duality pairing $H^{2}_0(\Omega)$-$H^{-2}(\Omega)$.
We denote   by $\nabla^k f$ the application of $k$ successive gradients, by $\Delta f $ the laplacian and by $\Delta^2 f:=\Delta\Delta f $ the  bilaplacian of $f$, respectively.
We indicate with  $[f, g] := \tfrac{\partial^2 f}{\partial x_1^2} \tfrac{\partial^2 g}{\partial x_2^2} + $ $\tfrac{\partial^2 f}{\partial x_2^2} \tfrac{\partial^2 g}{\partial x_1^2} $ $-2 \tfrac{\partial^2 f}{\partial x_1 \partial x_2} \tfrac{\partial^2 g}{\partial x_1 \partial x_2}$ the Monge-Amp\`ere operator. Observe that, by introducing the $2\times 2$ cofactor matrix, we also have $[f, g] =\cof(\nabla^2 f) : \nabla^2 g$,
where $:$ is used to express the Frobenius inner product. It is immediate to verify that $[f,f] = 2\text{det}(\nabla^2 f)$, and hence $[f,f](x)$ is twice the Gaussian curvature of the surface $(x, f(x))$ at $x \in \Omega$. \color{black}
We denote with  $(\cdot, \cdot)_{\Delta}$  the operator acting from $H^2(\Omega) \times H^2(\Omega) \to \mathbb{R}$ such that
\begin{equation}
	(f, g)_{\Delta} := \int_{\Omega} \Delta f \Delta g \, dx.
\end{equation}
Finally, we recall the fundamental relation between $\sigma$, the $2\times 2 $ mechanical stress tensor and the Airy stress function $v$ given by
 $\sigma = \cof(\nabla^2 v)$. Componentwise, this reads  $\sigma_{11}=v_{x_2x_2}$,
 $\sigma_{22}=v_{x_2x_2}$ and $\sigma_{12}=\sigma_{21}=-v_{x_1x_2}$.
 Additional notation used throughout the paper is summarized in Table \ref{table:2024}.

\begin{table}[h!]
	\centering
	\label{table:notation}
	\begin{tabular}{llc}
		\toprule
		\textbf{Symbol}                                           & \textbf{Description}                         & \textbf{Units}  \\
		\midrule
		$\Omega $                                                 & Reference configuration                      & m$^2$           \\
		$x = (x_1, x_2)$                                          & Point in $\Omega$                            & m               \\

		$w, w_o $                                                 & Transverse displacement, scale gauge         & m               \\
		$v, v_o $                                                 & Airy stress function, scale gauge            & Nm              \\

		$s,s_i $                                                  & Frank angle                                  & rad             \\
		$y^{(i)}$                                                 & Position of disclination                     & m               \\
		$\mathsf E$                                               & Young modulus                                & Pa              \\
		$\nu \in (-1, \frac{1}{2})$                               & Poisson ratio                                & non-dimensional \\

		$h $                                                      & Plate thickness                              & m               \\
		$\mathsf D := \tfrac{\mathsf E}{1-\nu^2} \tfrac{h^3}{12}$ & Plate bending stiffness                      & Nm              \\
		$c_\nu := \tfrac{1}{12(1-\nu^2)}$                         & Dimensionless bending stiffness              & {non-dim.}      \\
		$p, p_o $                                                 & Transverse load per unit area, scale gauge   & Pa              \\

		$R $                                                      & Characteristic size of $\Omega$              & m               \\
		$\beta := \tfrac{R}{h} $                                  & Plate aspect ratio                           & non-dim.        \\
		$\gamma := \tfrac{p_o}{\mathsf E} $                       & Transverse load to stiffness ratio           & non-dim.        \\
		\bottomrule
	\end{tabular}
	\caption{Material and geometric parameters.}
	\label{table:2024}
\end{table}

\newcommand{\Fnum}{\mathscr{F}_{\eta}}

\section{Mathematical analysis of FvK equations}\label{2501041121}
In this section, we present the governing equations for the  F\"oppl-von K\'arm\'an
theory in the context of kinematically incompatible thin plates subjected to external transverse loads, and we provide existence and regularity theorems for their solutions. Throughout the paper we assume that $\Omega $ is a domain, that is,   an open, bounded  and simply connected set in $\mathbb{R}^2$.
Recall that the Dirac delta $\delta \in H^{-2}(\Omega)$ thanks to
the Sobolev embedding theorem.
A finite collection of disclinations are located at points $y^{(i)}\in\Omega$, each characterized by its own angle $0\neq s_i \in \mathbb{R} $. 
Define  $\theta \in H^{-2}(\Omega)$ as
\begin{equation}
	\label{202412111457}
	\theta := \sum^N_{i=1} s_i \delta(x-y^{(i)})
\end{equation}
where $s_i \in \mathbb{R}\setminus \{0\}$ for every $i \in \{1, \hdots, N\}$.
Let $v, w \in H^2_0(\Omega)$.	The system
\begin{equation} \label{2406171748}
	\begin{cases}
		\begin{alignedat}{4}
			& D \Delta^2 w       & \; = \;& [v, w] + p & \quad & \text{in } H^{-2}(\Omega), \\
			& \frac{1}{E h} \Delta^2 v & \; = \;& -\frac{1}{2}[w, w] + \theta & \quad & \text{in } H^{-2}(\Omega),
		\end{alignedat}
	\end{cases}
\end{equation}
is the F\"oppl-von K\'arm\'an  model for kinematically incompatible thin plates (see \cite{SN88}, \cite{gao14} \cite{singh21}). 
It is shown in \cite{CDLM24} (see Proposition 1.7), in particular that, when $v$ is regular enough,the condition $v \in H^2_0(\Omega)$ guarantees that the membrane displacement satisfies homogeneous Neumann boundary conditions on $\partial\Omega$, (see also \cite{Lewicka11}). Consequently, the condition $v, w \in H^2_0(\Omega)$ corresponds to a plate whose boundary is free of in-plane traction and the out-of-plane displacement is clamped. 
The following theorems guarantee the existence and regularity of solutions to the system \eqref{2406171748},   in the presence of Dirac deltas for the membrane equation.
\begin{theorem}{(Existence)}
	\label{2408121907}
	Let $\Omega \subset \mathbb{R}^2$ be a domain with Lipschitz boundary. Let $\theta \in H^{-2}(\Omega)$ be defined as in \eqref{202412111457} and $p \in H^{-2}(\Omega)$. Then, Eq. \ref{2406171748} admits a solution in $H^2_0(\Omega) \times H^2_0(\Omega)$.
\end{theorem}

\begin{theorem}{(Regularity)}	\label{Thm:202420121439}
	Consider the assumptions of Theorem \ref{2408121907} and assume further that $p \in L^k(\Omega)$ for some real $k \ge 2$,
	$\partial \Omega \in C^{4,\gamma}$ for some $\gamma \in (0,1)$
	Denote with
	$(w, v) \in H^2_0(\Omega) \times H^2_0(\Omega)$ a solution to Eq. \ref{2406171748}.
	Then $w \in H^2_0(\Omega) \cap W^{4,k}(\Omega)$ and $v \in H_0^2(\Omega) \cap W^{2, \tau}(\Omega) \cap C^{4, \gamma}(\overline{\Omega} \setminus \cup_{i=1}^{N} \overline{B_{r}(y^{(i)})})$ for any $r > 0$ such that $\overline{B_{r}(y^{(i)})} \subset \Omega$ ($i \in \{1, \hdots, N\}$)
	 and any $\tau \in [1, \infty)$.
\end{theorem}
The proof of existence and regularity are an adaptation of the technique presented by Ciarlet for the case $s_i = 0$.
As we have not been able to locate a proof available in the literature for this specific setting,  we report the details in an Appendix for the reader's convenience.
Notice that, as customary for FvK equations in the compatible case, one cannot in general expect uniqueness of the solution.
We highlight that the regularity theorem has significant practical utility for our simulations, as it ensures that, thanks to the Sobolev embedding, \( w \in C^{2,\zeta}(\overline{\Omega}) \) with \( \zeta \in (0,1) \). Hence, even in the presence of rough/wrinkled solutions, the deflection is at least twice differentiable.

\section{Variational formulation}
\label{sec:202501011737}
We show that Equation \eqref{2406171748} can be derived via the variational principle for a suitable functional.

\begin{theorem}\label{2412291203}
	Under the assumptions of Theorem \ref{2408121907},
	the system of equations \eqref{2406171748}	corresponds 	to the 	first variation of
	the functional
	$		\mathscr{F} : H^2_0(\Omega) \times H^2_0(\Omega) \to \mathbb{R}  $ defined as
	\begin{gather}
		\label{eq:202412181200}
		\mathscr{F}(v, w) := -\mathscr{E}_{m}(v) + \mathscr{E}_{c}(v, w) + \mathscr{E}_{b}(w) + \langle v, \theta \rangle - \langle w, p \rangle,\textrm{ where}\\
		\mathscr{E}_{m}(v) := \frac{1}{2Eh} \int_{\Omega} |\nabla^2 v|^2 \, dx, \quad
		\mathscr{E}_{c}(v, w) :=  	\frac{1}{2}	\int_{\Omega} \cof(\nabla^2 v) : \nabla w \otimes \nabla w \, dx,
		\quad \mathscr{E}_{b}(w) := \frac{D}{2} \int_{\Omega} |\nabla^2 w|^2 \, dx.\nonumber
	\end{gather}
\end{theorem}
Before proving the theorem, we show in a  remark   an alternative representation of the functional $\mathscr{F} $ for the clamped case.
\begin{remark}\label{remark:202412111625}
	Functional  $\mathscr{F}$ defined in Eq. \ref{eq:202412181200}
	is equal to
	\begin{equation}\label{2412291252}
		-\frac{1}{2Eh} \int_{\Omega} |\nabla^2 v|^2 \, dx   -\frac{1}{2} \int_{\Omega} [w,w] \, v \, dx + \frac{D}{2} \int_{\Omega} |\nabla^2 w|^2 \, dx+ \langle v, \theta \rangle - \langle w, p \rangle
	\end{equation}
	in $H^2_0(\Omega) \times H^2_0(\Omega)$.
	The claim follows by Lemma \ref{lemma:202412131500} given in the Appendix by taking $\phi = v$, $\eta = \chi = w$
	in Eq. 	\ref{eq:202412131450}.
%
	%
%
Eq. \ref{2412291252} provides an alternative expression for the functional \eqref{eq:202412181200}. 
We use Eq. \eqref{eq:202412181200} to derive the set of discretized equations. 
We use  \eqref{2412291252}  to establish the strong FvK field equations.
	However, when other boundary conditions are considered the two expressions may not coincide.
\end{remark}

\paragraph{Proof of Theorem \ref{2412291203}.}
Using the representation formula in Eq. \eqref{2412291252}, we compute
	\begin{equation}
		\frac{d}{dt} \left( -\frac{1}{2} \int_{\Omega} [w,w]  (v + t \varphi) \, dx \right) \Big|_{t=0} = -\frac{1}{2} \int_{\Omega} [w,w]  \varphi \, dx \quad \forall \varphi \in H_0^2(\Omega).
	\end{equation}
	and
	\begin{equation}
		\frac{d}{dt} \left( -\frac{1}{2} \int_{\Omega} [w + t \varphi, w + t \varphi]  v \, dx \right) \Big|_{t=0} = -\int_{\Omega} [\varphi, w]  v \, dx = -\int_{\Omega} [v,w]  \varphi \, dx \quad \forall \varphi \in H_0^2(\Omega)
	\end{equation}
%
where in the least equality we have used the symmetry (see \cite[Theorem 5.8-2]{ciarlet97})
\begin{equation}
	\int_{\Omega} [\varphi, w] v \, dx = \int_{\Omega} [w, v]  \varphi \, dx.
\end{equation}
The computation of the first variations of $\mathscr{E}_{m}$, $\mathscr{E}_{b}$, $\langle v, \theta \rangle$ and $\langle w, p \rangle$ is straightforward.
$\square$

\subsection{Non-dimensional formulation}
\label{subsection:2416101524}
We highlight the role of key parameters by rescaling dependent and independent variables as follows
\begin{equation}
	\xi := \frac{x}{R} \qquad \tilde p(\xi) := \frac{p(x/R)}{p_o} \qquad \tilde w(\xi) := \frac{w(x/R)}{w_o} \qquad \tilde v(\xi) := \frac{v(x/R)}{v_o}
\end{equation}
where $p_o$, $w_o$ and $v_o$ are   scaling factors to be tuned. We denote by $\tilde \Omega$ the domain $\Omega$ rescaled by its characteristic size $R$, and by $[\cdot, \cdot]_{\xi} := R^4 [\cdot, \cdot]$ and $\nabla^2_{\xi} := R^2 \nabla^2$, the rescaled Monge-Amp\`ere bracket and Hessian operator, respectively. Upon rescaling, functional \eqref{eq:202412181200} reads
\begin{multline}
	\label{2412291317}
	\mathscr{F}(v, w) =
	-\frac{v_o^{2}}{E h R^{2}} \frac{1}{2} \int_{\tilde \Omega} |\nabla_{\xi}^2 \tilde  v|^2  d\xi + \frac{E h^{3} w_o^{2}}{R^{2}} \frac{c_\nu}{2} \int_{\tilde \Omega} |\nabla_{\xi}^2 \tilde w|^2  d\xi - \frac{v_o w_o^{2}}{2 R^{2}} \int_{\tilde \Omega} [\tilde w, \tilde w]_{\xi}  \tilde v  d\xi + \\
	- R^{2} p_o w_o \int_\Omega \tilde p  \tilde w  d\xi
	+ v_o \langle \tilde \theta, \tilde v \rangle
\end{multline}
where $\tilde \theta := \sum\limits_{i=1}^N s_i \delta(\xi - \tilde{y}^{(i)})$, $\tilde{y}^{(i)} := y^{(i)}/R$,
and we have made use of the explicit expressions for
$D$ and $c_\nu$
from Table \ref{table:2024}.
By rescaling Eq. \ref{2412291317} by $\frac{v^2_o}{E R^2 h}$, we identify the rescaled functional
$ \frac{E R^2 h}{v^2_o}	\mathscr{F}=:\mathscr{\tilde F}: H^2_0(\tilde \Omega) \times H^2_0(\tilde \Omega)\mapsto\mathbb{R}$
\begin{multline}
	\label{eq:202412111724}
	\mathscr{\tilde F}(\tilde v, \tilde w) :=
	- \frac{1}{2}\int_{\tilde \Omega} |\nabla_{\xi}^2 \tilde v|^2  d\xi + \frac{E^2 h^4 w_o^2}{v_o^2} \frac{c_\nu}{2} \int_{\tilde \Omega} |\nabla_{\xi}^2 \tilde w|^2  d\xi - \frac{E h w_o^2}{v_o} \frac{1}{2} \int_{\tilde \Omega} [\tilde w, \tilde w]_{\xi}  \tilde v  d\xi + \\
	- \frac{p_o w_o E R^4 h}{v^2_o} \int_{\tilde \Omega} \tilde p  \tilde w  d\xi + \frac{E R^2 h}{v_o} \langle \tilde \theta, \tilde v \rangle.
\end{multline}
We choose $w_o$ and $v_o$ so that the coefficients of the bending and coupling energy terms in \eqref{eq:202412111724} are simultaneously equal to $1$. By solving the system
\begin{equation}
	\label{eq:202412111723}
	\frac{E^2 h^4 w_o^2}{v_o^2} = 1, \qquad \frac{E h w_o^2}{v_o} = 1,
\end{equation}
we obtain $w_o = h$ and $v_o = E h^3$, which in turn implies that $\mathscr{E}_m$, $\mathscr{E}_b$ and $\mathscr{E}_c$ scale as $\frac{Eh^5}{R^2}$. To simplify the notation, we drop the tildas and the subscript $\xi$ 
so that the non-dimensional energy finally reads
\begin{equation}
	\label{eq:202412161314}
	\mathscr{F}(v, w) :=
	- \frac{1}{2} \int_{\Omega} |\nabla^2 v|^2 \, d\xi + \frac{c_\nu}{2} \int_{\Omega} |\nabla^2 w|^2 \, d\xi
	- \frac{1}{2} \int_{\Omega} [w, w]  v \, d\xi
	- \gamma \beta^4 \int_{\Omega} p  w \, d\xi + \beta^2 \langle \theta, v \rangle.
\end{equation}
Non-dimensional parameters
\begin{equation} \label{eq:202501011638}
	\beta := \frac{R}{h}, \qquad \gamma := \frac{p_o}{E}
\end{equation}
are the plate aspect ratio (slenderness),  and the scaled pressure, which measures the transverse load relative to the (three-dimensional) material stiffness.
The first variation of \eqref{eq:202412161314} yields
\begin{equation} \label{2408121717}
	\begin{cases}
		\begin{alignedat}{4}
			& c_\nu \Delta^2 w & \; = \;& [v, w] + \gamma\beta^4  p & \quad & \text{in } H^{-2}(\O), \\
			& \hspace{0.6em}\Delta^2 v & \; = \;& -\frac{1}{2}[w, w] + \beta^2  \theta & \quad & \text{in } H^{-2}(\O),
		\end{alignedat}
	\end{cases}
\end{equation}
that is, the non-dimensional form of \eqref{2406171748}.
In \eqref{2408121717}-(2) remark that $\beta^2$ acts as a geometric gain factor for the distribution of disclinations, while in \eqref{2408121717}-(1) $\beta^4$ scales the transverse external load.

To estimate physically relevant values  for the non-dimensional parameters, consider a thin plate of characteristic length $R=1$, thickness $h = 0.01$m, Young's modulus $E = 10$GPa, and Poisson ratio $\nu = 0.15$, subject to a transverse pressure due to a distributed load comparable to its own weight. Assuming a density $\rho \sim 10^3$kg/m$^3$ (representative of materials like steel as well as graphene) and denoting by $g$ the gravitational acceleration, the characteristic pressure is $p \sim \rho g h \sim 100$Pa.
We then obtain $	\overline{\beta} = 100, \overline{\gamma} = 10^{-8}$ as the reference values  for the non-dimensional parameters.
The range of values for the aspect ratio $\beta$ and the relative transverse pressure $\gamma$ that we will consider in the following numerical experiments span the physical ranges of moderately thick ($\beta \sim 10$) to very thin ($\beta \sim 100$) plates, and transverse loads starting from $10$-fold the plate's own weight.

\subsection{Discrete Numerical Formulation}\label{sec:IPCDformulation}
For our fourth-order problem with $H^2$-scalar unknowns, we use $C^0$-continuous elements and apply an interior penalty method to weakly enforce gradient continuity across mesh interfaces (see \cite{Brenner}). This results in a non-conforming Discontinuous Galerkin (DG) method ~\cite{Arnold82, Arnold01}.
A comprehensive comparison of DG methods for fourth-order problems can be found in \cite{Arnold00}, while implementations for biharmonic problems are discussed in \cite{Georgoulis08, Dong23}. For the specific case of F\"oppl-von K\'arm\'an equations,  $C^0$-DG methods and corresponding error estimates are described in \cite{Carstensen}.

Our implementation builds upon the formulation in $\cite{brunetti}$. Instead of adding the penalty terms to the discrete weak form, we generalize the approach of $\cite{brunetti}$ by incorporating the penalties directly into the energy functional as additional energy terms.
This leads to a more compact form for the overall functional and provides greater flexibility in the implementation.

Let  $\mathcal{T}$  denote a regular triangulation of  $\Omega \subset \mathbb{R}^2$ into closed triangles. Define \( \edgbd \) as the set of boundary edges, \( \edgin \) as the set of internal edges, so that \( \mathcal{E} := \edgin \cup \edgbd \) is the collection of all edges. For  $T \in \mathcal{T}$, let  $\mathrm{P}_3(T)$  be the space of polynomials of degree at most three defined on  $T$. The Lagrange Finite Element space  $\mathcal{V}(\mathcal{T})$  consists of globally continuous, piecewise cubic polynomials that vanish on the boundary edges \( \edgbd \), namely
\begin{equation}
	\mathcal{V}(\mathcal{T}) := \Big\{ u \in C^{0}(\Omega) : u|_{T} \in \text{P}_3(T) \quad \forall T \in \mathcal{T} \Big\} \cap H^1_0(\Omega)
\end{equation}
%
We following \cite{Brenner:Von_Karman}
and add additional terms to our
 DG formulation   accounting for discontinuities in the gradient across mesh interfaces and ensure the coercivity of the problem in $\mathcal{V}(\mathcal T)$.
For every edge $e \in \mathcal{E}$   and every scalar field $u : \overline{\Omega} \to \mathbb{R}$, the jump and average operators are defined as follows
\begin{equation}\label{2412291421}
	\llbracket u \rrbracket_{e} :=
	\begin{cases}
		u_{+}|_{e} - u_{-}|_{e} & \text{ if } e \in \mathcal{E}_{\text{in}}  \\
		u|_{e}                  & \text{ if } e \in \mathcal{E}_{\text{bnd}}
	\end{cases}
	\qquad
	\{ u \}_{e} :=
	\begin{cases}
		\displaystyle{\frac{u_{+}|_{e} + u_{-}|_{e}}{2} } & \text{ if } e \in \mathcal{E}_{\text{in}}  \\
		u|_{e}                                            & \text{ if } e \in \mathcal{E}_{\text{bnd}}
	\end{cases}
\end{equation}
where $u_{+} = u |_{T^{+}}$, $u_{-} = u |_{T^{-}}$, $T^{+}$,
$T^{-}$ are two closed triangles whose intersection is   $e$.
(Definitions \ref{2412291421} should be applied componentwise when $u$ is a vector or a matrix). 
Finally, for any $T \in \mathcal{\mathcal{T}}$ let $\eta := \text{diam}(T)$.
We   take $n_e$ to be the outward unit normal on $e \in \mathcal{E}_{\text{in}}$, pointing from $T^{-}$ to $T^{+}$. Note that $\llbracket \partial_{n_e} u \rrbracket$ and $\{ \partial_{n_e n_e} u \}$ are independent on the particular labelling of two adjacent triangles $T^{+}$ and $T^{-}$. Based on the definition \eqref{eq:202412161314} and in view of   Remark \ref{remark:202412111625} we define $\Fnum : \mathcal{V}(\mathcal{T}) \times \mathcal{V}(\mathcal{T}) \to \mathbb{R} $ as
\begin{multline}
	\label{eq: Functional Von-Karman Plate discretized}
	\Fnum(\uv, \uw) := \hspace{-0.4em} \sum\limits_{T \in \mathcal{T}} \Big( -\frac{1}{2} \lVert \nabla^2 \uv \rVert^2_{L^2(T)} + \frac{1}{2} \int_{T} \cof(\nabla^2 \uv) : (\nabla \uw \otimes \nabla \uw) \, d\xi + \frac{c_\nu}{2} \lVert \nabla^2 \uw \rVert^2_{L^2(T)} + \\
	-\gamma\beta^4 \int_T \hspace{-0.4em} \uw  p \, d\xi  \Big)  + \beta^2 \langle \uv , \theta \rangle
\end{multline}
and
\begin{equation}
	\label{eq: DG formulation - LCDG1}
	\mathscr{L}_{1}(u) := -\frac{1}{2} \sum_{e \in \mathcal{E}} \int_{e} \llbracket \nabla u \cdot n_e \rrbracket_{e}  \{ \nabla^2 u \hspace{0.2em} n_e \cdot n_e \}_{e} \, d\mathcal{H}^1 \qquad \mathscr{L}_{2}(u) := \frac{1}{2} \sum_{e \in \mathcal{E}} \frac{\alpha}{\{ \eta \}_e} \int_{e} \llbracket \nabla u \cdot n_e \rrbracket_{e}^2 \, d\mathcal{H}^1
\end{equation}
where $\mathscr{L}_1, \mathscr{L}_2 : \mathcal{V}(\mathcal{T}) \to \mathbb{R} $.

The  first variation of $\mathscr{L}_1$ gives rise to the classical terms in the symmetric DG formulation of the biharmonic problem. The key novelty in our approach is that we prescribe the DG formulation at the energy level, rather than in a weak form.
The term	$\mathscr{L}_2$ is a standard finite penalization term that penalizes the jump in the normal direction of the gradient, ensuring that the  solution is globally $C^1$.
This is achieved because the tangential component of the gradient is continuous, as  we use $C^0$ elements. The parameter $\alpha$ is not part of the mechanical model but is a tunable parameter, which we set to 300 in our simulations.

We are now in a position to introduce our symmetric Interior Penalty DG formulation of problem \eqref{2408121717} based on the variational principle. Define $\mathscr{I} : \mathcal{V}(\mathcal{T}) \times \mathcal{V}(\mathcal{T}) \to \mathbb{R} $
where
\begin{align}
	\mathscr{I}_\eta(\uv, \uw) := \Fnum(\uv, \uw) + c_\nu \mathscr{L}_{1}(\uw) + \mathscr{L}_{2}(\uw) -  \mathscr{L}_{1}(\uv) - \mathscr{L}_{2}(\uv) + \\\notag 
	- \frac{1}{2} \sum_{e \in \mathcal{E}_{\text{in}}} \int_{e} \{ \cof(\nabla \uw \otimes \nabla \uw) : n_e \otimes n_e \}_{e} \llbracket \nabla \uv \cdot n_e \rrbracket_{e} \, d\mathcal{H}^1.
\end{align}
We seek $(v, w)$ $\in \mathcal{V}(\mathcal{T}) \times \mathcal{V}(\mathcal{T})$ such that the following equations are simultaneously verified
	\begin{align}
		\label{eq:202412181350}
		\dfrac{d \mathscr{I}_\eta(v + t \varphi , w)}{dt} \Big|_{t = 0} = 0, \qquad
		\frac{d\mathscr{I}_\eta(v, w + t \varphi )}{dt} \Big|_{t = 0} = 0 \qquad \forall \varphi \in \mathcal{V}(\mathcal{T}).
	\end{align}
The discrete weak DG formulation, which we refer to as VAR (Variational) in what follows, reads as follows:
%
find $(v, w)$ $\in \mathcal{V}(\mathcal{T}) \times \mathcal{V}(\mathcal{T})$ such that
\newcommand{\myup}{\blacksquare}
\newcommand{\mydown}{\blacksquare_2}
\newcommand{\nyup}{\blacktriangledown}
\newcommand{\nydown}{\triangledown}
\begin{multline}
	\label{eq:202412271313}
	- \sum\limits_{T \in \mathcal{T}} \left( \int_{T} \nabla^2 v : \nabla^2 \varphi \, d\xi + \frac{1}{2} \int_{T} \underbrace{\left( \nabla w \otimes \nabla w \right) :  \cof(\nabla^2 \varphi) }_{\myup} \, d\xi \right) + \sum_{e \in \mathcal{E}} \Big( \int_{e} \llbracket \nabla v \cdot n_e \rrbracket_{e}  \{ \nabla^2 \varphi  n_e \cdot n_e \}_{e} \, d\mathcal{H}^1 + \\
	+ \int_{e} \llbracket \nabla \varphi \cdot n_e \rrbracket_{e}  \{ \nabla^2 v  n_e \cdot n_e \}_{e} \, d\mathcal{H}^1 - \frac{\alpha}{\{ \eta \}_e} \int_{e} \llbracket \nabla v \cdot n_e \rrbracket_{e} \llbracket \nabla \varphi \cdot n_e \rrbracket_{e} \, d\mathcal{H}^1 + \\
	- \frac{1}{2} \sum_{e \in \mathcal{E}_{\text{in}}} \int_{e} \underbrace{ \{ \cof(\nabla w \otimes \nabla w) : n_e \otimes n_e \}_{e} \llbracket \nabla \varphi \cdot n_e \rrbracket_{e}}_{\nyup} \, d\mathcal{H}^1 \Big) = - \beta^2 \sum\limits_{T \in \mathcal{T}} \int_T \varphi  \theta \, d\xi, \quad \forall \varphi \in \mathcal{V}(\mathcal{T});
\end{multline}
\begin{multline}
	\label{eq:202412271314}
	\sum\limits_{T \in \mathcal{T}} \left( c_{\nu} \int_{T} \nabla^2 w : \nabla^2 \varphi \, d\xi + \int_{T} \underbrace{ \left( \nabla w \otimes \nabla \varphi \right) :  \cof(\nabla^2 v) }_{\Box}  d\xi \right) + \sum_{e \in \mathcal{E}} \Big( - c_{\nu} \int_{e} \llbracket \nabla w \cdot n_e \rrbracket_{e}  \{ \nabla^2 \varphi  n_e \cdot n_e \}_{e} \, d\mathcal{H}^1 + \\
	- c_{\nu} \int_{e} \llbracket \nabla \varphi \cdot n_e \rrbracket_{e}  \{ \nabla^2 w \hspace{0.2em} n_e \cdot n_e \}_{e} \, d\mathcal{H}^1 + \frac{\alpha}{\{ \eta \}_e} \int_{e} \llbracket \nabla w \cdot n_e \rrbracket_{e} \llbracket \nabla \varphi \cdot n_e \rrbracket_{e} \, d\mathcal{H}^1 + \\  -  \sum_{e \in \mathcal{E}_{\text{in}}} \int_{e} \underbrace{ \{ \cof(\nabla w \otimes \nabla \varphi) : n_e \otimes n_e \}_{e}  \llbracket \nabla v \cdot n_e \rrbracket_{e} }_{\nydown} \, d\mathcal{H}^1  \Big)
	= f \sum\limits_{T \in \mathcal{T}} \int_{T} \varphi  p \, d\xi,  \quad \forall \varphi \in \mathcal{V}(\mathcal{T}).
\end{multline}
To assess the performance of our approach, we compare it with two alternative formulations available in the literature. The first formulation (\cite{Brenner:Von_Karman}) can be derived from   \eqref{eq:202412271313} and \eqref{eq:202412271314}, where the terms marked by
$\myup$ and $\Box$, are replaced with
$-[w, w] \varphi $ and $ -[v, w] \varphi$, respectively, and the terms marked by  $\nyup$, $\nydown$ are replaced by 
$$-\llbracket \{ \cof(\nabla^2 w) \}_e \nabla w \cdot n_e \rrbracket_e  \varphi \quad \text{ and } \quad  -\frac{1}{2} \left( \llbracket \{ \cof(\nabla^2 v) \}_e \nabla w \cdot n_e \rrbracket_e  +  \llbracket \{ \cof(\nabla^2 w) \}_e \nabla v \cdot n_e \rrbracket_e  \right)  \varphi, 
$$ 
respectively. The second formulation (\cite{Carstensen}) is obtained by replacing $\myup$ and $\Box$ as above and by omitting the terms $\nyup$ and $\nydown$.




\section{Numerical results}\label{2501041129}

We solve the  variational  problem \eqref{eq:202412181350} numerically using the  DolfinX/FEniCSx framework (\cite{BarattaEtal2023}) for Finite Element discretization, and PETSc (\cite{petsc-tao-users-manual}) for efficient and scalable linear algebra operations.  Our implementation (available at \url{https://github.com/kumiori/disclinations}) uses Lagrange  finite elements,
with the order of the elements configurable at runtime
to provide flexibility in the discretization of the equations. 
Specifically, we employ third-order Lagrange elements for both the transverse displacement and the Airy  potential, resulting in a piecewise linear approximation for the Gaussian curvature.
For the nonlinear problem, we utilize the SNES solver provided by PETSc.  The iterative subproblems are solved by using a semi-direct Krylov subspace method, with LU factorization for the preconditioning step.
In all our simulations the domain $\Omega$ is a unit disc centered at the origin, discretized using triangular (Delaunay) elements.
For a mesh size of $\eta= 0.02$ (which is relevant for the parametric simulation of Section \ref{2501041148}), the mesh consists of approximately 9k nodes and 19k elements.
This yields stiffness matrices of size 17k$\times$17k elements and 13m 
nonzero entries ($0.05\%$of the total).

The non-dimensional formulation involves two tunable parameters: the aspect ratio $\beta$ and $\gamma$, the relative transverse pressure to the Young’s modulus (see Eq. \eqref{eq:202501011638}).
We restrict the disclination  Frank angle to the characteristic values
$\{\pm \tfrac{1}{2},\pm 1\}$. This choice enables a  clear distinction between the positive and negative stress fields induced by positive and negative disclinations, corresponding to tension and compression.
All computations in this section use the non-dimensional formulation.

\subsection{Solver verification}\label{2501041124}

Two verification experiments test the reliability of the numerical framework with respect to exact solutions.
The first involves a plate without disclinations deformed under a transverse load, the second features no external loads, but an incompatibility caused by a pair of  symmetric disclinations.
We consider a very thin plate with $\beta=100$ and a mesh size $\eta = 0.05$, with 1.6k nodes and 3k elements.

\subsubsection{Test 1: kinematically compatible plate under transverse pressure}\label{sec:test1}

Consider Eq. \eqref{2408121717} with $\theta(\xi) = 0$ (no disclinations) and an applied transverse load
\begin{equation}
p(r) = \sqrt{2c_{\nu}^3} \left(\frac{40}{3} (1 - r^2 )^4 + \frac{16}{3} (11 + r^2 ) \right)
\end{equation}
where $r := |\xi|$.
The transverse load induces out-of-plane deflection, which activates  membrane deformations of elastic nature through the coupling.
The pair $(\wexact, \vexact)$,
\begin{equation}
\label{eq:202501011729}
\wexact(r) := \sqrt{2c_{\nu}} (1 - r^2 )^2, \qquad \vexact(r) := c_{\nu} \left(-\frac{1}{12} (1 - r^2 )^2 -\frac{1}{18} (1 - r^2)^3 -\frac{1}{24} (1 - r^2)^4 \right)
\end{equation}
is a known solution to Eq. \eqref{2408121717} (see \cite{Reiser}). In Figures \ref{fig:241091104}-LEFT and \ref{fig:241091104}-RIGHT we show profiles of the numerical solutions obtained from the three FE formulations described in Section $\ref{sec:IPCDformulation}$ for $\xi_2 = 0$.

\begin{figure}[h!]
\centering
\begin{minipage}[b]{0.45\textwidth}
	\centering
	\includegraphics[width=\textwidth]{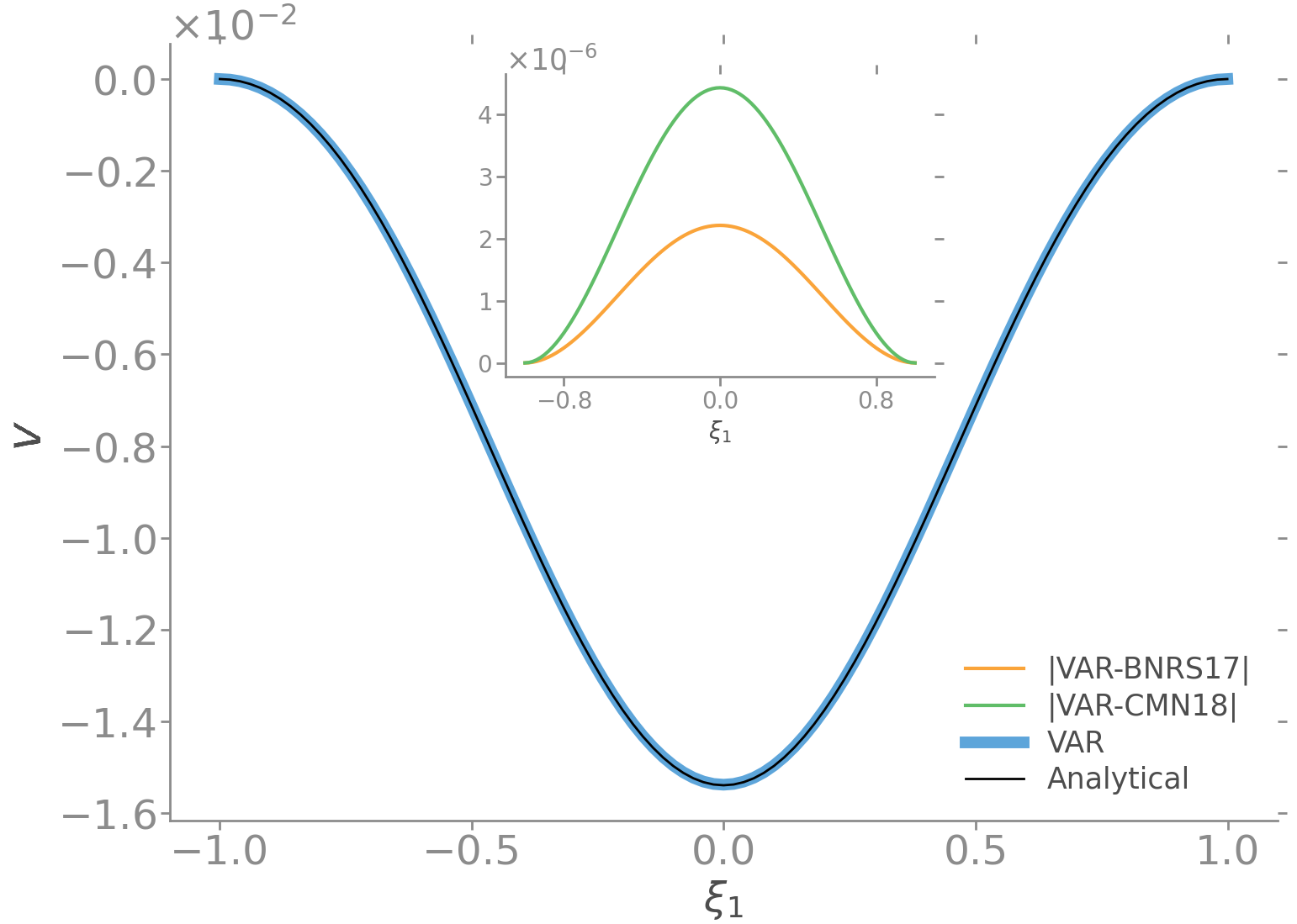} 
\end{minipage}
\hfill
\begin{minipage}[b]{0.45\textwidth}
	\centering
	\includegraphics[width=\textwidth]{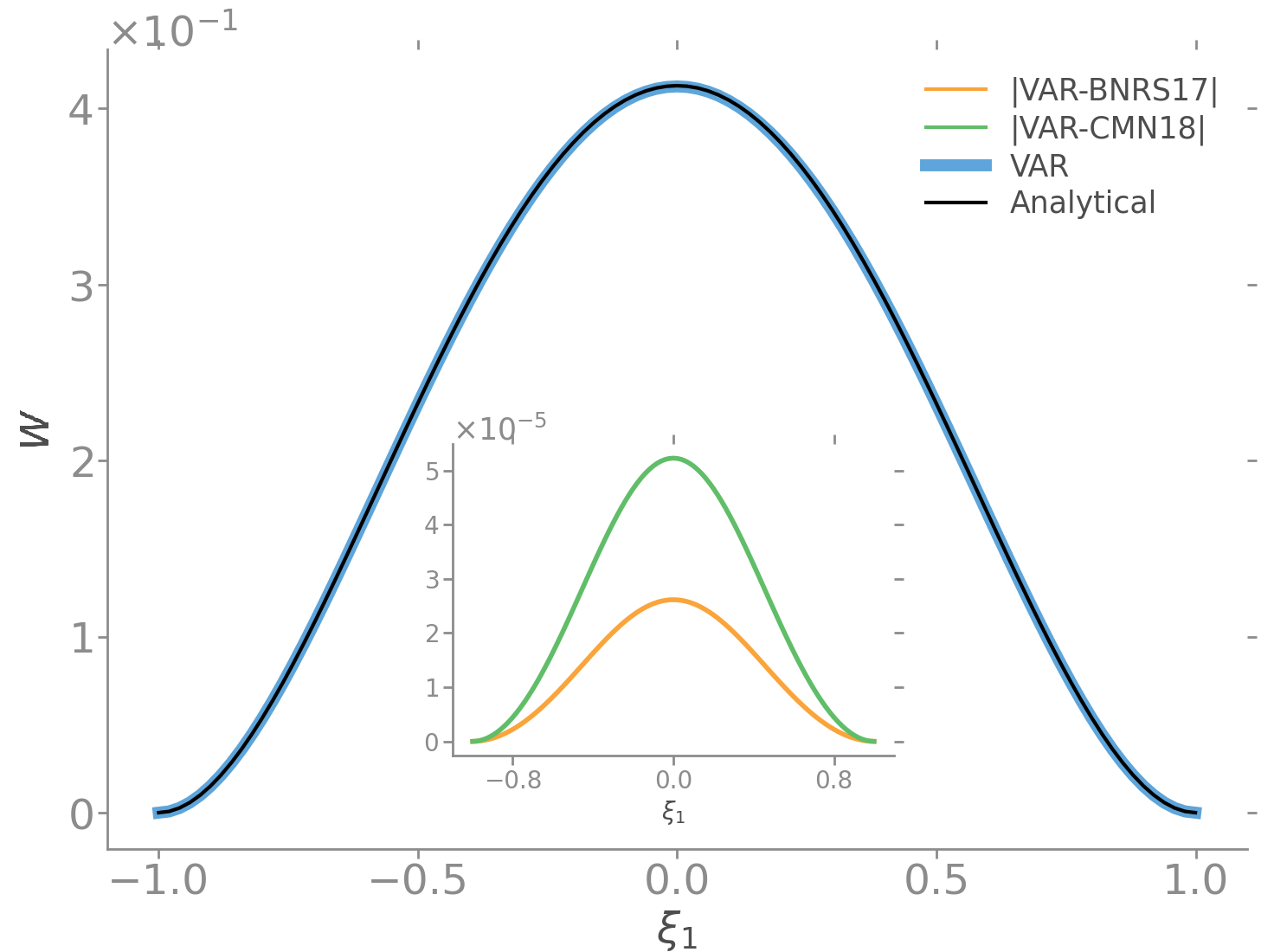} 
\end{minipage}
\caption{Profiles of $v(\xi_1, 0)$ (LEFT) and $w(\xi_1, 0)$ (RIGHT) for $\xi_1 \in [-1, 1]$. Variational formulation (VAR) as in Eqs. \ref{eq:202412271313} and \ref{eq:202412271314} (blue thick line), analytical solution (black thin line). In the inset, the difference between: VAR and BNRS17 as in \cite{Brenner:Von_Karman} (orange line), VAR and CMN18 as in \cite{Carstensen} (green line).}
\label{fig:241091104}
\end{figure}

We observe no significant deviations between the three FE formulations and the analytical solution.
For details, see Table \ref{2501042255} 
which presents the  relative errors for the bending and membrane energies and for the coupling term.

\begin{table}[H]
	\centering
	\begin{tabular}{lccc}
		\toprule
		\textbf{\qquad} & VAR (\eqref{eq:202412271313},
		\eqref {eq:202412271314}) & \cite{Brenner:Von_Karman} & \cite{Carstensen} \\
		\midrule
		$e_b$ & -0.085 \% & -0.074\% & -0.064\% \\
		$e_m$ & -0.632\% & -0.661\% & -0.690\% \\
		$e_c$ & -0.426\% & -0.428\% & -0.43\% \\
		\bottomrule
	\end{tabular}
	\caption{The energy percentage errors are computed as $\text{e}_X := 100 \times \frac{\mathscr{E}_X - \mathscr{E}_{X,\text{exact}} }{\mathscr{E}_{X,\text{exact}}}$, where $X$ stands for $b$, $m$ or $c$, corresponding, respectively to the bending, membrane and coupling integral terms ($\mathscr{E}_{b}$, $\mathscr{E}_{m}$ and $\mathscr{E}_{c}$) defined in Section \ref{sec:202501011737}. These errors are evaluated based on the three numerical solutions and the corresponding analytical solution \eqref{eq:202501011729}.}
	\label{2501042255}
\end{table}


\subsubsection{Test 2: kinematically incompatible plate with two disclinations}
We consider a circular plate with two symmetric disclinations of opposite Frank angles, modeled by 
Eq. \eqref{2408121717} with $\theta = \delta(\xi-y^{(1)})$ $- \delta(\xi+y^{(1)})$, where 
$y^{(1)} = (0.2, 0)$, and $p = 0$.
In this configuration, non-zero membrane strains and inelastic stresses are induced by the presence of the disclinations. The pair $(\wexact, \vexact) $ given by
\begin{equation}
\wexact(\xi) = 0 \qquad \vexact(\xi) := \beta^2 \mathscr{G}(\xi; y^{(1)}) -\beta^2 \mathscr{G}(\xi;-y^{(1)})
\end{equation}
solves Eq. \eqref{2408121717}.
Here, for $y$ fixed in $\Omega$
\begin{equation}
\label{eq:202412271400}
\mathscr{G}(\xi; y) := \begin{cases}
\displaystyle \frac{1}{16\pi} \left(\left(1-\xi^2\right) \left(1-y^2\right)+|\xi-y|^2 \displaystyle \ln  \frac{|\xi-y|^2}{\xi^2 y^2 - 2\xi \cdot y + 1} \right) \quad & \text{for } \xi \ne y \\
0 \quad                                                                                                                                                               & \text{otherwise}
\end{cases}
\end{equation}
is  known from the theory of Green's functions (see \cite{PolyharmonicGreenFunction}).
A comparison of the numerical solutions, based on the three FE implementations is reported in Figure \ref{fig:241091050}, showing excellent agreement with the exact solution.
We omit the profile of the 	transverse displacement $w$, as it is found to be zero in all   three implementations.
The relative error for the membrane energy is $-0.518\%$ for all the three FE formulations described in Section $\ref{sec:IPCDformulation}$.

\begin{figure}[h!]
\centering
\includegraphics[width=0.5\textwidth]{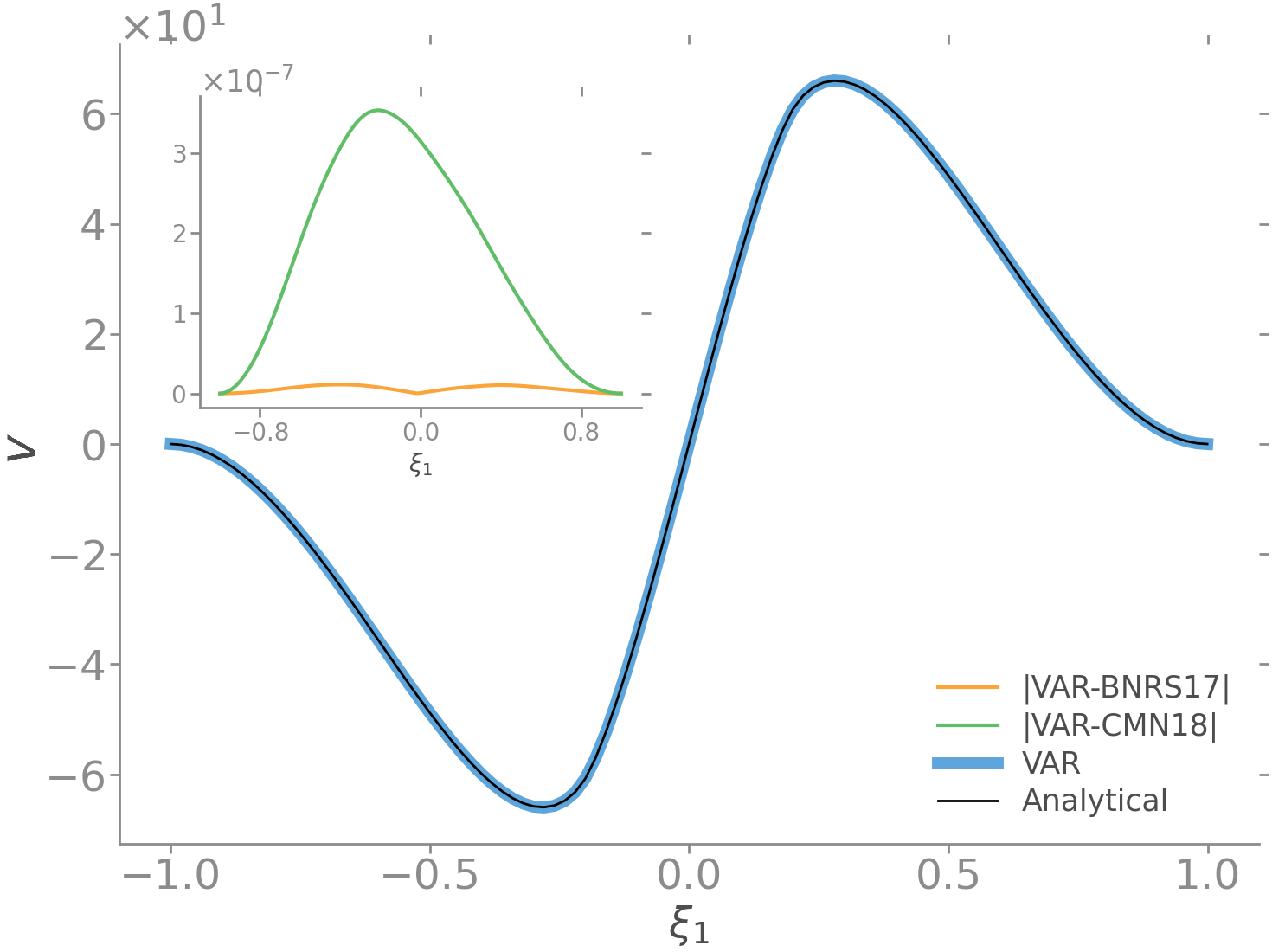} 
\caption{Profile of $v(\xi_1, 0)$ for $\xi_1 \in [-1, 1]$. Variational formulation (VAR) as in Eqs. \ref{eq:202412271313} and \ref{eq:202412271314} (blue thick line), analytical solution (black thin line). In the inset, the difference between: VAR and BNRS17 as in \cite{Brenner:Von_Karman} (orange line), VAR and CMN18 as in \cite{Carstensen} (green line).}

\label{fig:241091050}
\end{figure}

\subsection{Parametric study varying slenderness and load ratio}\label{2501041148}

We perform sets of numerical simulations varying the parameters $\beta$ and $\gamma$.
The purpose   is to explore the non-linear coupling in the FvK equations across various regimes in the presence of disclinations.
First, we vary $\beta$  (plate aspect ratio) while keeping $\gamma \beta^4$ (the prefactor of the transverse external load) constant.
Secondly, we vary $\gamma$ keeping $\beta$ constant.
The first experiment corresponds to studying the mechanics of the plate by varying the influence of the disclination for a fixed external load.
The second corresponds to varying the external load while maintaining the effect of the disclination constant.

Throughout this section we place one disclination at the origin with Frank angle $s = -1$ and consider $p(\xi) = -1$.

\subsubsection{The effect of the aspect ratio $\beta$}
\label{sec:202412091330}
We fix $\gamma \beta^4 = 1$ while varying the plate aspect ratio $\beta$ in the range $[10, 100]$, 
plotting the scaling of the different  energy components in Fig. \ref{fig:24108179}.
%
%
%
The membrane energy grows like $\beta^4$, as a consequence of a partial decoupling of the membrane equation. In this regime, the out-of-plane deflection as well as the  bracket $[w, w] $ are small compared to the other terms in the equation for the stress potential.
This is clearly illustrated in  Fig. \ref{fig:24108179}.  Neglecting $ [w, w] $ in  
the second equation of the system   $\ref{2408121717}$, the membrane equation reduces to
\begin{equation} \label{2408121226}
\Delta^2 v =   -\beta^2\delta(\xi) \quad \text{in }  \O.\\
\end{equation}
In Eq.  $\ref{2408121226},$ the membrane stress is autonomous, and can be immediately solved using its own Green’s function	$H^2_0(\Omega)\ni \hat v =-	\beta^2 \mathscr{G}(\xi; 0)$, where $\mathscr{G}$ is given in Eq. $\ref{eq:202412271400}$.
%
%
%
Eq. \ref{2408121226} corresponds to the  Kirchhoff-Love (KL) membrane equation 
for a kinematically incompatible thin plate	\cite{Kirchhoff-Love}. 
In  Fig. $\ref{fig:24108179}$ we show
the 	profile of the membrane energy computed using the KL model (based on the solution to Eq. \ref{2408121226}) and the membrane energy   of the full FvK model, showing excellent   agreement.
For completeness, in Fig. \ref{fig:24108179} we also present
a comparison between the bending energy computed according to the KL model and the full FvK model.  The energy   according to the KL model is given by the solution to the following KL-type equation  	
\begin{equation} \label{2501171042}
c_\nu \Delta^2 w =  - \gamma \beta^4 \quad \text{in }  \O.\\
\end{equation}
The profiles do not match because the effect of the disclination, via the coupling term $[v,w]$ is neglected in \ref{2501171042}. However, this effect is  not small.
The effect of the negative disclination is to reduce the bending energy by flattening the profile of $w$.

\begin{figure}[h!]
\centering
\includegraphics[width=0.7\textwidth]{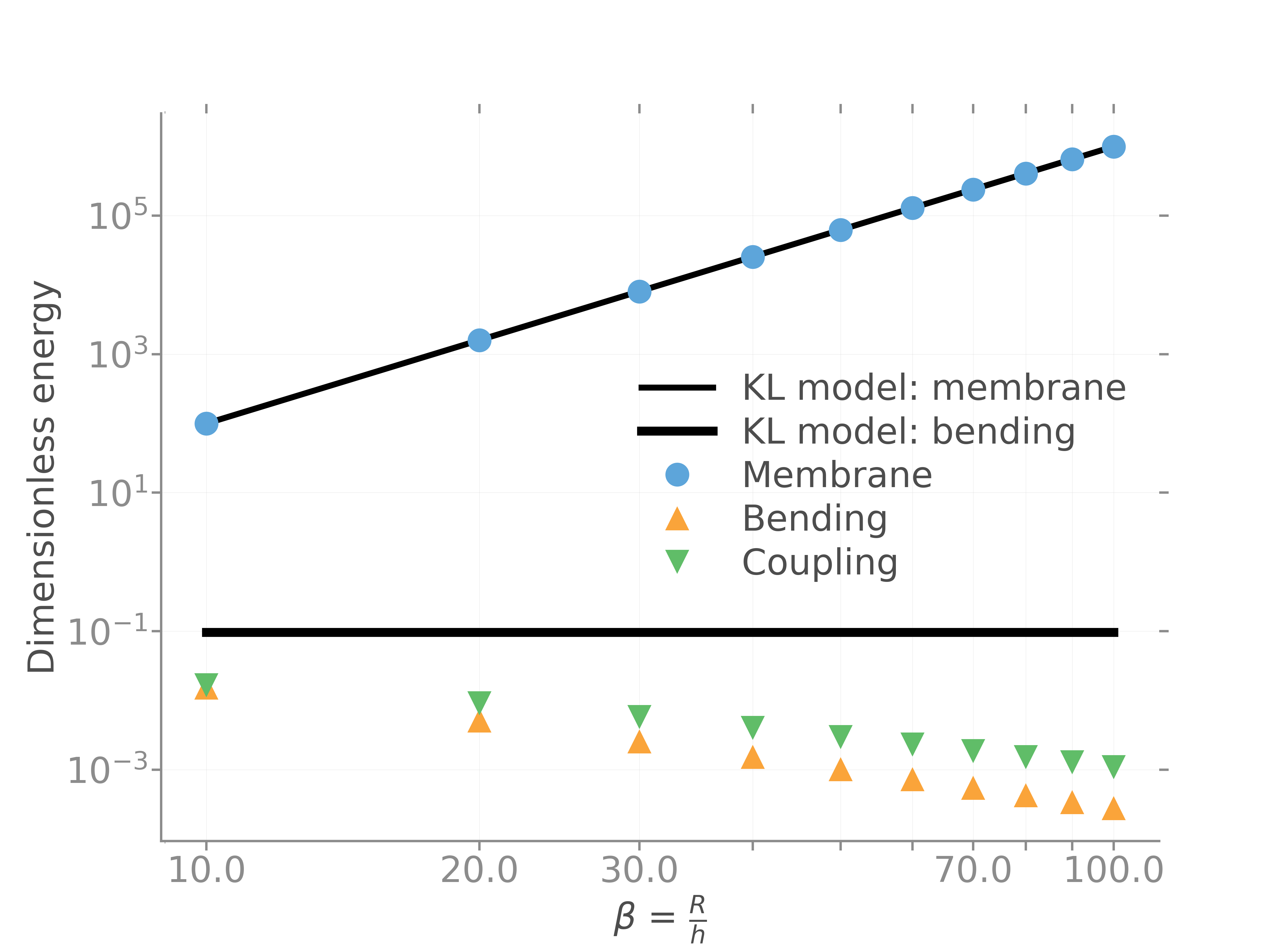} 
\caption{Membrane and bending energies and coupling term. Membrane (blue circles). Bending (orange triangles). Coupling (green  triangles). Membrane energy according to Kirchhoff-Love can be computed exactly and is equal to $\mathscr{E}_{\text{m}}^{\text{KL}} = \frac{\beta^4}{32\pi} $ (black thin line). Bending energy predicted by the Kirchhoff-Love theory,  $\mathscr{E}_{\text{b}}^{\text{KL}} = \frac{\pi}{384 c_{\nu}} (\gamma\beta^4)^2 $ (black thick line). Plot in log-log scale.}
\label{fig:24108179}
\end{figure}
Because the disclination angle $s$ is negative, the stress potential induces a radial tensile stress field.
Denoting by $e_r$ the radial unit vector, this is shown in Figure \ref{fig:2412092103}, where the radial stress field $\sigma_{rr} := \sigma e_r \cdot e_r$, computed for $\beta = 10$, is everywhere non-negative. 
\begin{figure}[h!]
\centering
\includegraphics[width=0.3\textwidth]{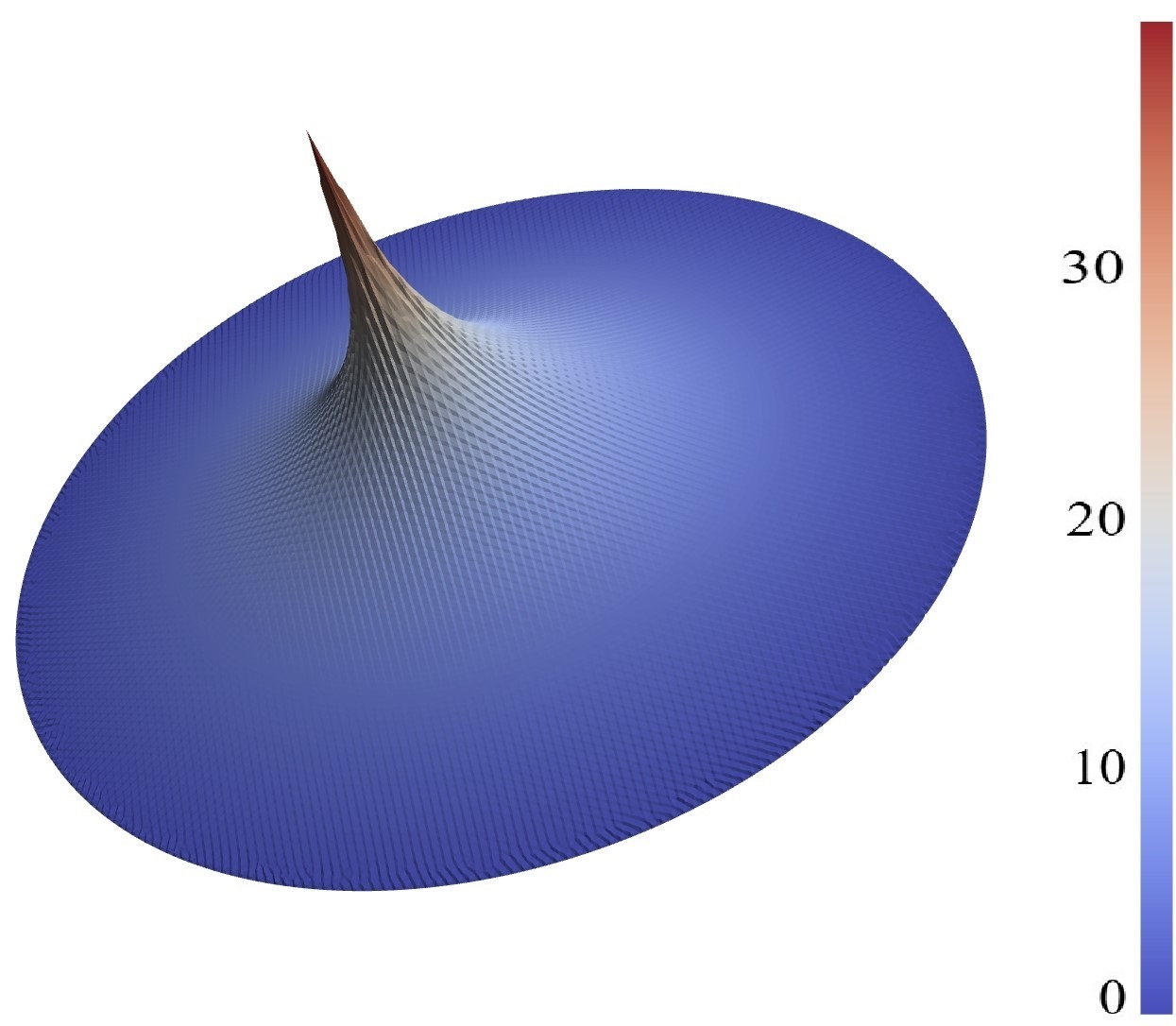} 
\caption{Dimensionless $\sigma_{rr}$ computed for $\beta = 10$.}
\label{fig:2412092103}
\end{figure}
This observation also justifies the scaling of the   bending and coupling terms in Fig. \ref{fig:24108179}.
Indeed, a plate subjected to a combination of a transverse load and a tensile membrane stress experiences smaller transverse deformation compared to the same plate without the tensile stress (see \cite{timoshenko1959theory}, Section 94, pages 391-393).
Hence, both the dimensionless bending and coupling integrals, which depend on the deflection, decrease as $\beta$ (which represents the influence  of the disclination) increases.

In Figure \ref{fig:202412171702}-LEFT and \ref{fig:202412171702}-RIGHT we report the profiles of, respectively, the dimensionless Airy stress function $v$ and transverse displacement $w$ obtained for $\beta = 10$ and $\beta = 100$. We can see that the profile of $v$ is self-similar with a scale factor $\beta^2$, confirming that the membrane equation is decoupled from the bending physics.
\begin{figure}[htbp]
\centering
\begin{minipage}{0.48\textwidth}
\centering
\includegraphics[width=\textwidth]{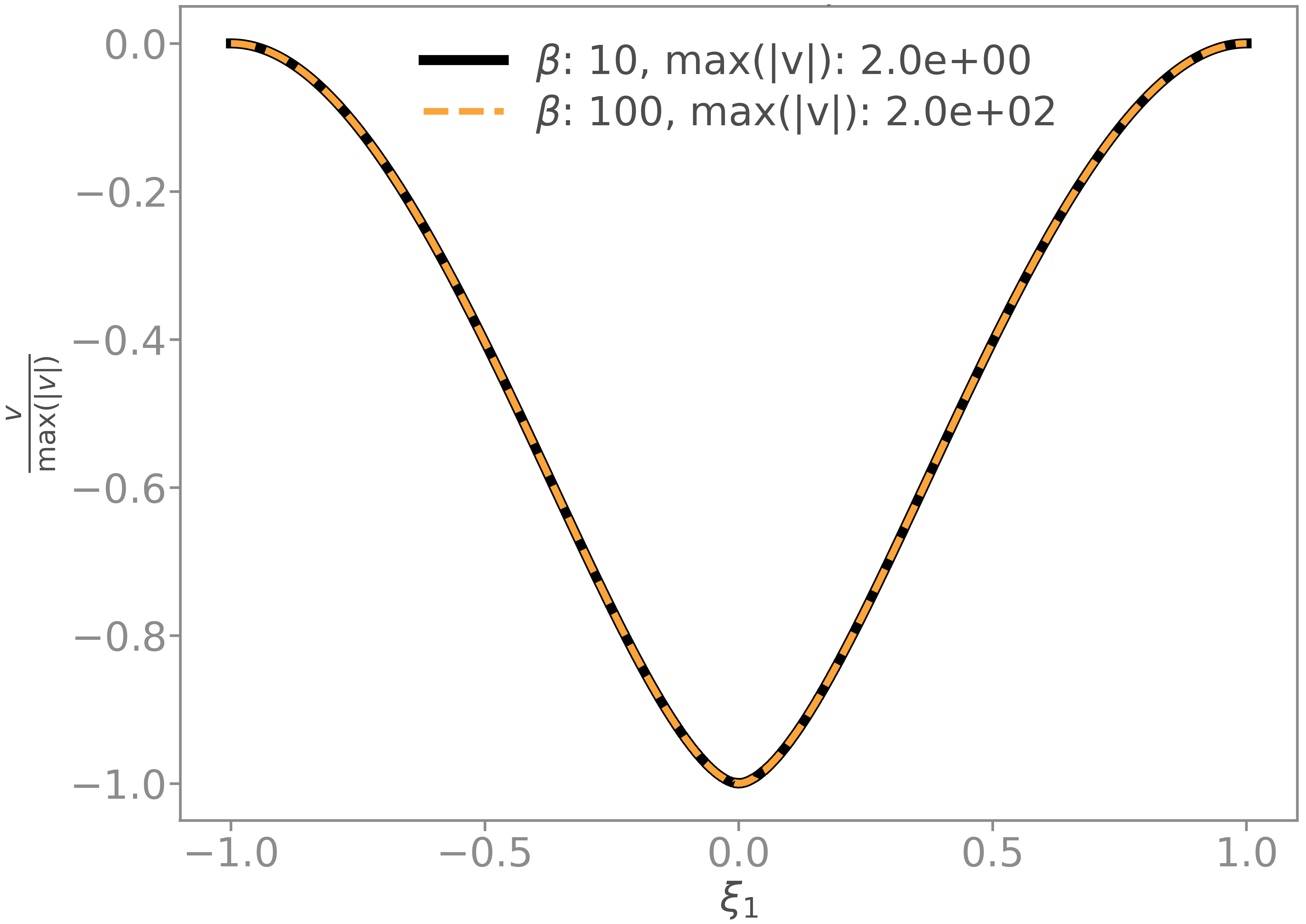} 
\end{minipage}
\hfill
\begin{minipage}{0.48\textwidth}
\centering
\includegraphics[width=\textwidth]{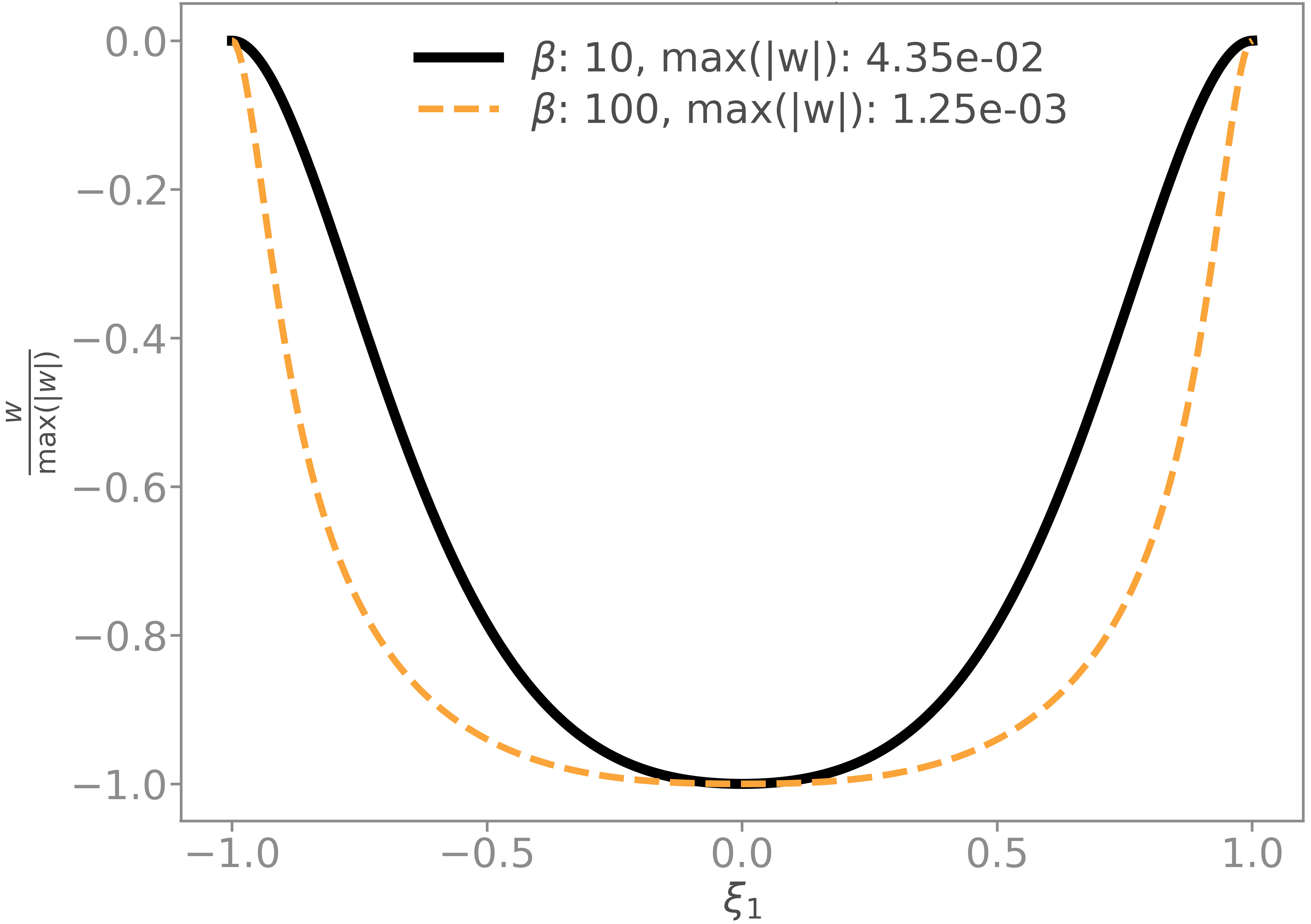} 

\end{minipage}
\caption{Profiles   of the normalized Airy potential (LEFT) and normalized transverse deflection (RIGHT). Functions evaluated for $\xi_1 \in [-1, 1]$ and $\xi_2 = 0$ and normalized with respect to their maximum value in the same range of $(\xi_1, \xi_2)$. $\beta = 10$ (black line), $\beta = 100$ (orange dashed line).}
\label{fig:202412171702}
\end{figure}
The transverse deformation, conversely, exhibits a plateau in the middle of the plate which extends with $\beta$ and corresponds to the region where the tensile stress field is strongest.

\subsubsection{The effect of the dimensionless external load $\gamma$}
We fix $\beta = 20$
and take $\gamma\in [6 \times 10^{-7}, 2\times 10^{-2}]$.
For small $\gamma \lesssim 2\times  10^{-3}$, both
coupling and   bending effects
are small (see Fig. \ref{fig:241081710}). Consequently, 
the full FvK membrane energy is well approximated by the  Kirchhoff-Love estimate, where $v$ is again given by
the solution to Eq. \ref{2408121226}.
In this regime the profile of the  membrane energy is flat,
because it is governed by the KL theory, calculated using the solution to Eq. \ref{2408121226} and  depends on $\beta$, not $\gamma$.
%
%
%
%
%
Conversely, the profiles for the bending energy calculated according to the FvK and KL models, as expected, do not match.  However they share a similar $\gamma^2$-dependence. The reason is as follows.
By replacing $v$ with $\hat v$   in the first equation of system \ref{2408121717}, the bending equation for FvK becomes linear in $\gamma$, and consequently,  the energy is quadratic in $\gamma$.
Likewise for the coupling term (suffices to substitute		$\hat v$).


\begin{figure}[h!]
\centering
\includegraphics[width=0.6\textwidth]{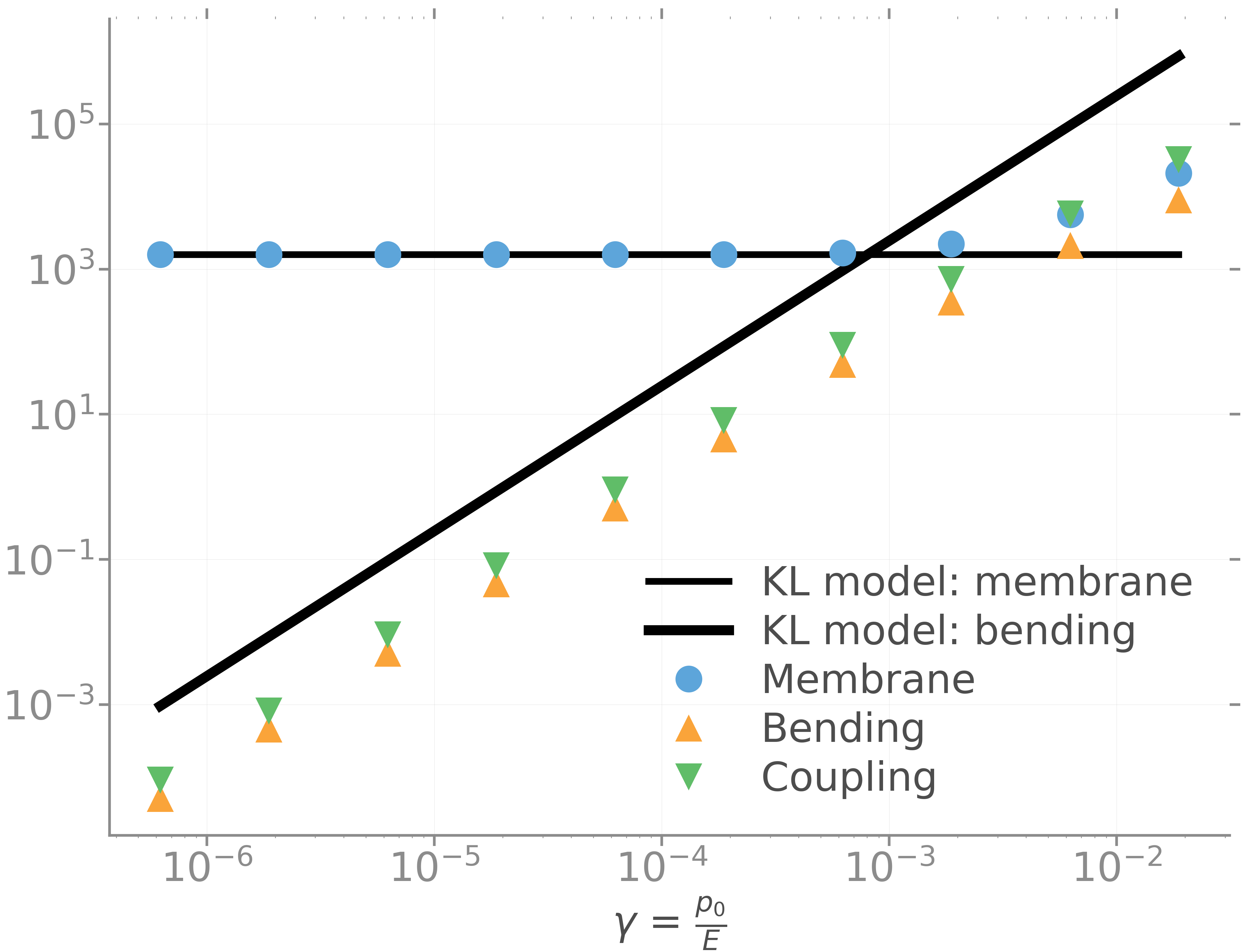} 
\caption{Dimensionless membrane and bending energies and coupling term. Membrane energy (blue circles), bending energy (orange triangles), coupling (green upside-down triangles), membrane energy calculated according to the KL theory $\mathscr{E}_{\text{m}}^{\text{KL}} = \tfrac{\beta^4}{32\pi} $ (black thin line), bending energy predicted by the KL theory $\mathscr{E}_{\text{b}}^{\text{KL}} = \tfrac{\pi}{384 c_{\nu}} (\gamma\beta^4)^2 $ (black thick line). Plot in log-log scale.}
\label{fig:241081710}
\end{figure}
For large $\gamma \gtrsim 2\times 10^{-3}$, the full non-linear coupling between the Airy potential and the deflection is activated, and consequently, the Kirchhoff-Love approximation is no longer valid.
This behavior can be observed in Figure \ref{fig:2412080521}-LEFT, where we report the profiles of the dimensionless $v$ obtained for three values of $\gamma$.
For small $\gamma=6.25\times  10^{-7}$ and $\gamma= 6.25\times 10^{-4}$
the shape of $v$ remains essentially unchanged. The absolute value of the stress potential increases by only about $2.5\%$.
Conversely, for a large value $\gamma = 1.875\times  10^{-2}$,
the stress potential is largely distorted by the large  Gaussian curvature of the plate measured by  $[w,w]/2$.

\begin{figure}[htbp]
\centering
\begin{minipage}{0.48\textwidth}
\centering
\includegraphics[width=\textwidth]{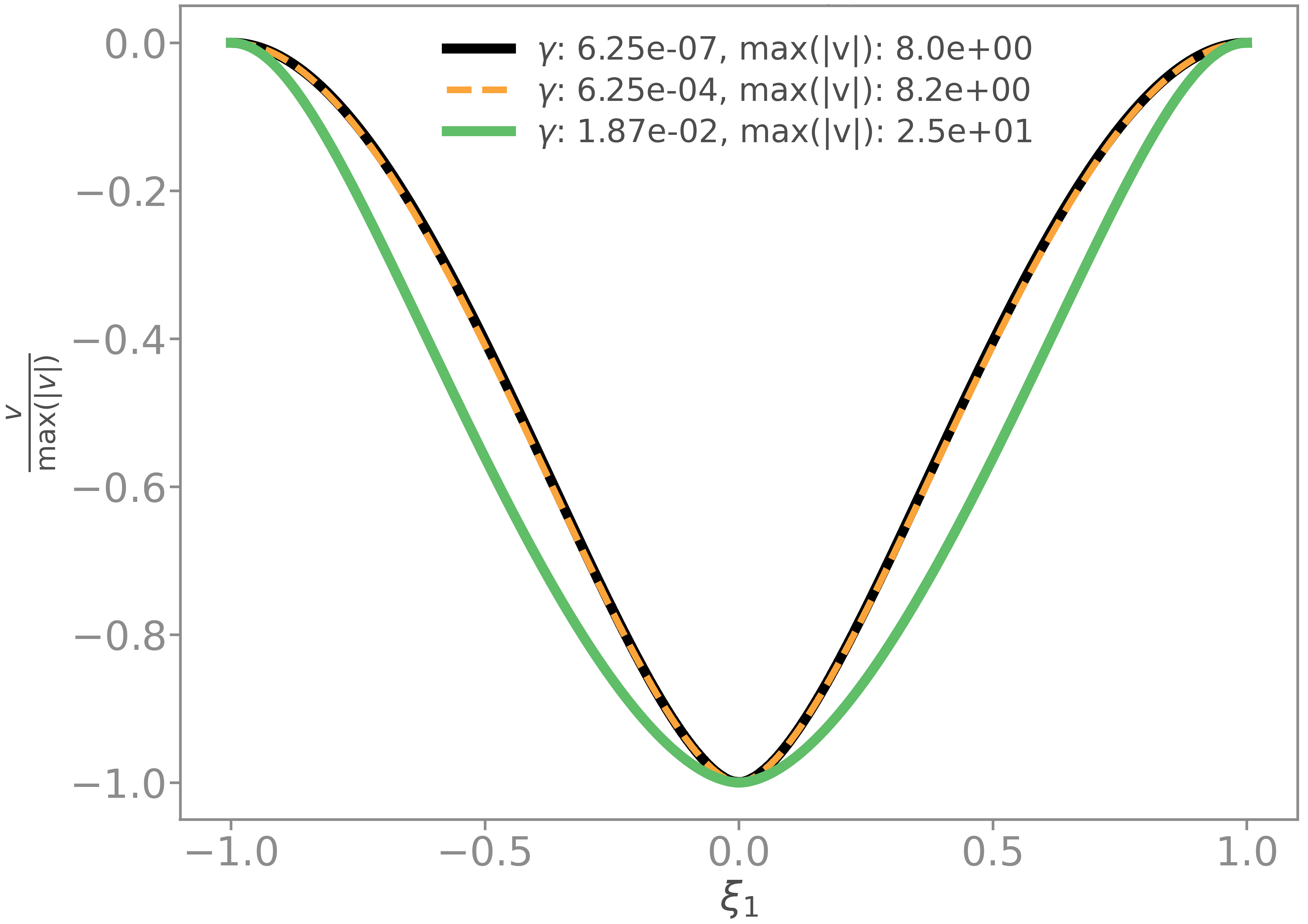} 
\end{minipage}
\hfill
\begin{minipage}{0.48\textwidth}
\centering
\centering
\includegraphics[width=\textwidth]{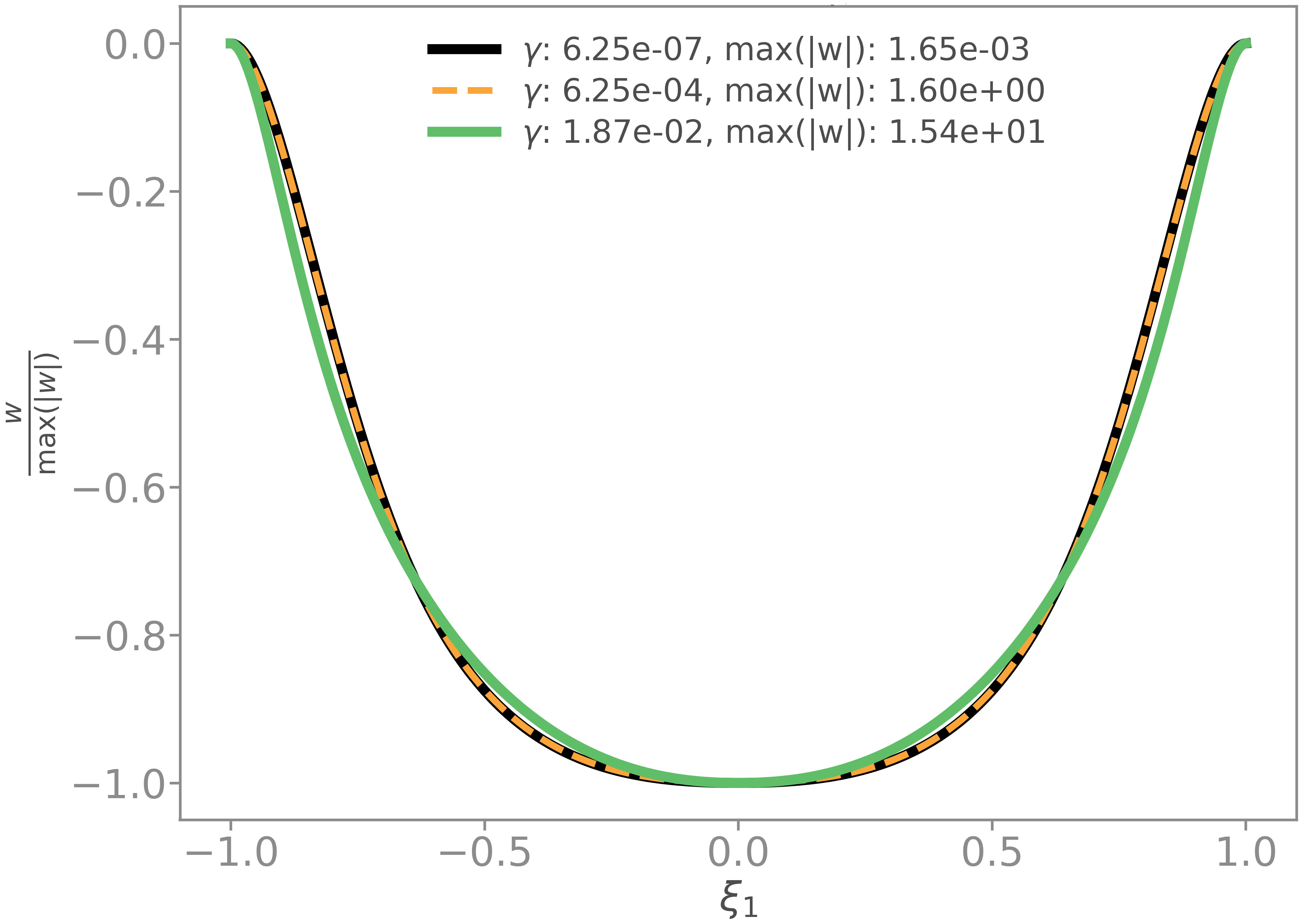} 

\end{minipage}
\caption{Profiles   of the normalized Airy potential (LEFT) and normalized transverse deflection (RIGHT). These functions are evaluated for $\xi_1 \in [-1, 1]$ and $\xi_2 = 0$ and normalized with respect to their maximum value within the same range of $(\xi_1, \xi_2)$. The following values of $\gamma$ are considered: $\gamma = 6.25\times 10^{-7}$ (black line), $\gamma = 6.25\times 10^{-4}$ (orange dashed line), $\gamma = 1.875\times 10^{-2}$ (green line).}
\label{fig:2412080521}
\end{figure}

In Figure \ref{fig:2412080521}-RIGHT we report the profile of $w$ obtained for the three different values of $\gamma$. In the region of small $\gamma$, the solutions are self-similar when rescaled by $\gamma$.
For large $\gamma$, this property, and consequently the validity of the linear approximation, no longer hold.

\subsection{Multi-disclination configurations}\label{2501041149}

Inspired by observations of complex   patterns in graphene, this section explores configurations where the presence of disclinations induces non-trivial Gaussian curvature in a circular plate.
In graphene, such configurations typically involve pairs of disclinations with opposite Frank angles placed at close distances to screen out their respective stress fields. These configurations include dipoles \cite{TSUCHIDA24} and Stone-Wales defects \cite{Kabir16}.
In the following simulations we assume  $\beta = 20$, $\gamma = 5 \times 10^{-8}$, $p(\xi) =-1$
and the mesh size $\eta=0.02$. 
\subsubsection{Non-zero total Frank angle}\label{}

We explore configurations with four disclinations of equal Frank angles, yielding high-stress solutions.  
This setup, although unrealistic for real materials, results in plates with non-trivial curvature patterns due to the nonlinear coupling.
%
We consider  four disclinations (each with an  equal Frank angle $s$, to be specified later)
arranged symmetrically according to the following geometry
\begin{equation}
\label{eq:202412091308}
{	\theta(\xi) = s \left( \delta(\xi-y^{(1)}) + \delta(\xi-y^{(2)}) + \delta(\xi-y^{(3)}) + \delta(\xi-y^{(4)}) \right)}
\end{equation}
%
with $y^{(1)} = (\frac{1}{2}, 0)$, $y^{(2)} = (0, \frac{1}{2})$, $y^{(3)} = (-\frac{1}{2}, 0)$, $y^{(4)} = (0, -\frac{1}{2})$ in a unit disk.

First we set $s = -\frac{1}{2}$ in  Eq.  \eqref{eq:202412091308}.
In Figures $\ref{fig:fournegativedisc}$ we display $\sigma_{rr}$, everywhere non-negative, indicating a radial tensile stress state (see discussion in Section \ref{sec:202412091330}).
\begin{figure}[htbp]
\centering
\includegraphics[width=0.8\textwidth]{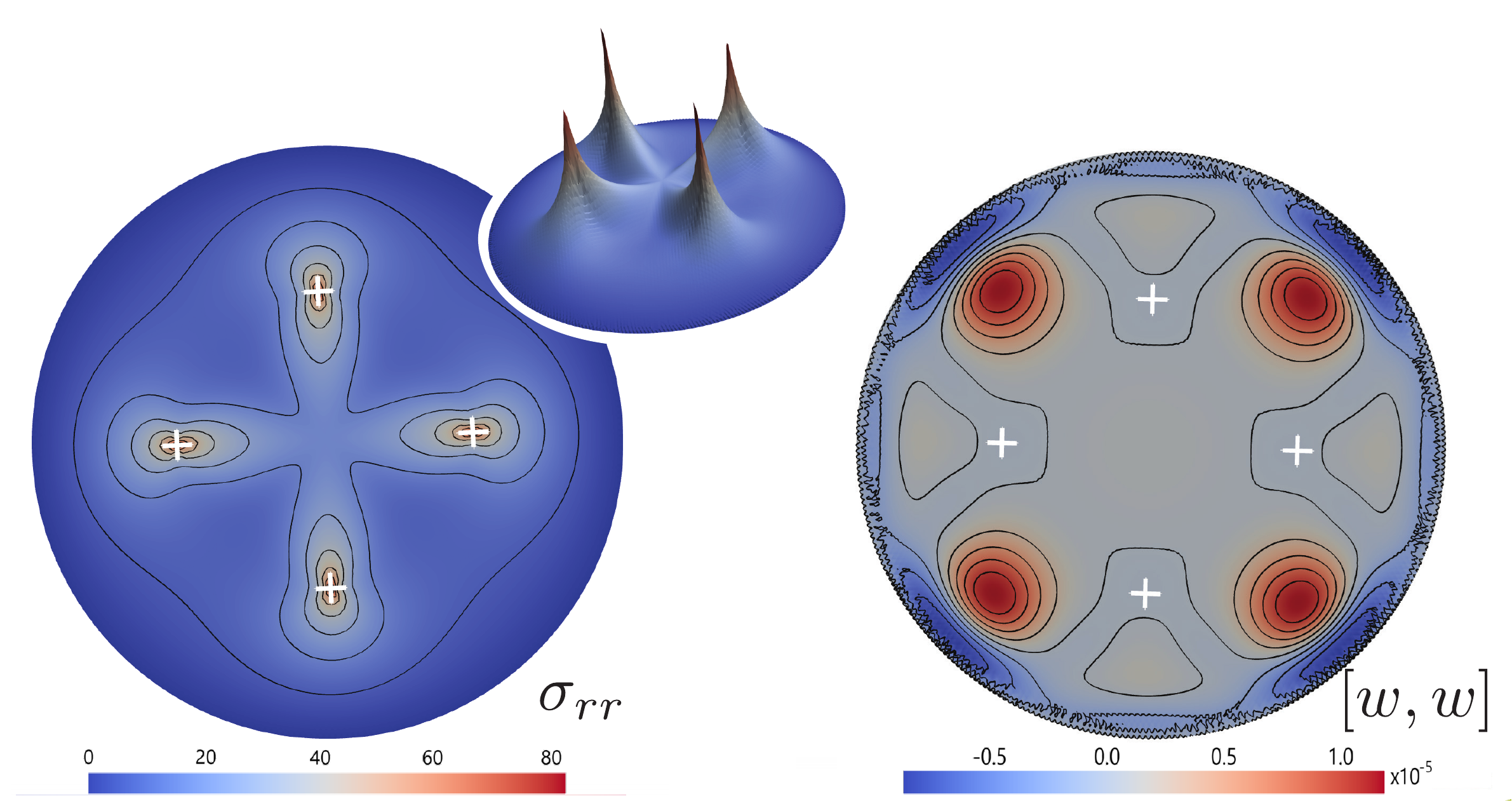} 
\caption{Configuration with four negative disclinations (white crosses) with $s=-\tfrac{1}{2}$. LEFT: radial stress $\sigma_{rr}$. Top and side views. RIGHT: level curves and heat map of $[w,w]$.}
\label{fig:fournegativedisc}
\end{figure}
\begin{figure}[htbp]
\centering
\begin{minipage}{0.4\textwidth}
\includegraphics[width=\textwidth]{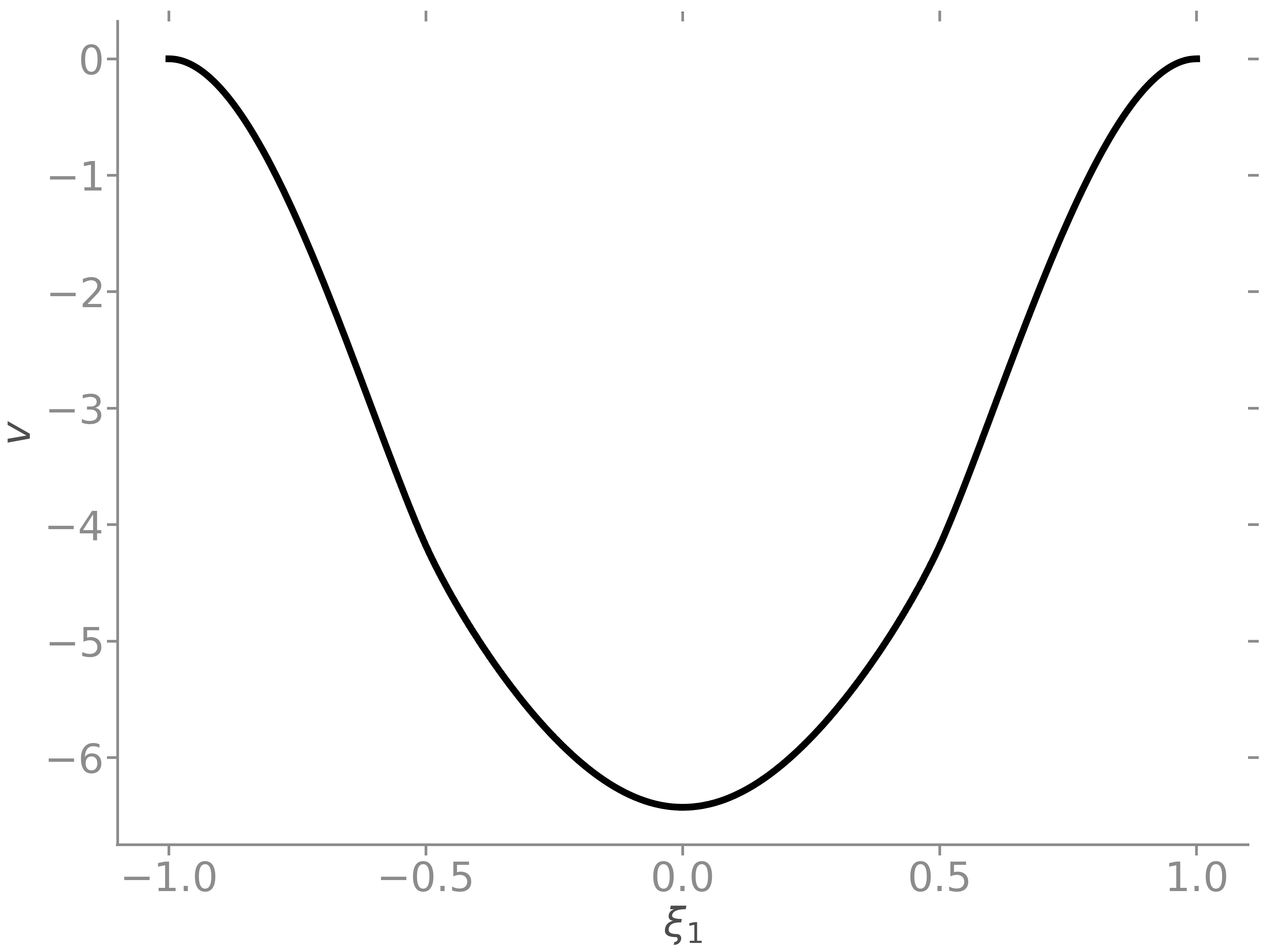}
\end{minipage}
\hfill
\begin{minipage}{0.4\textwidth}
\includegraphics[width=\textwidth]{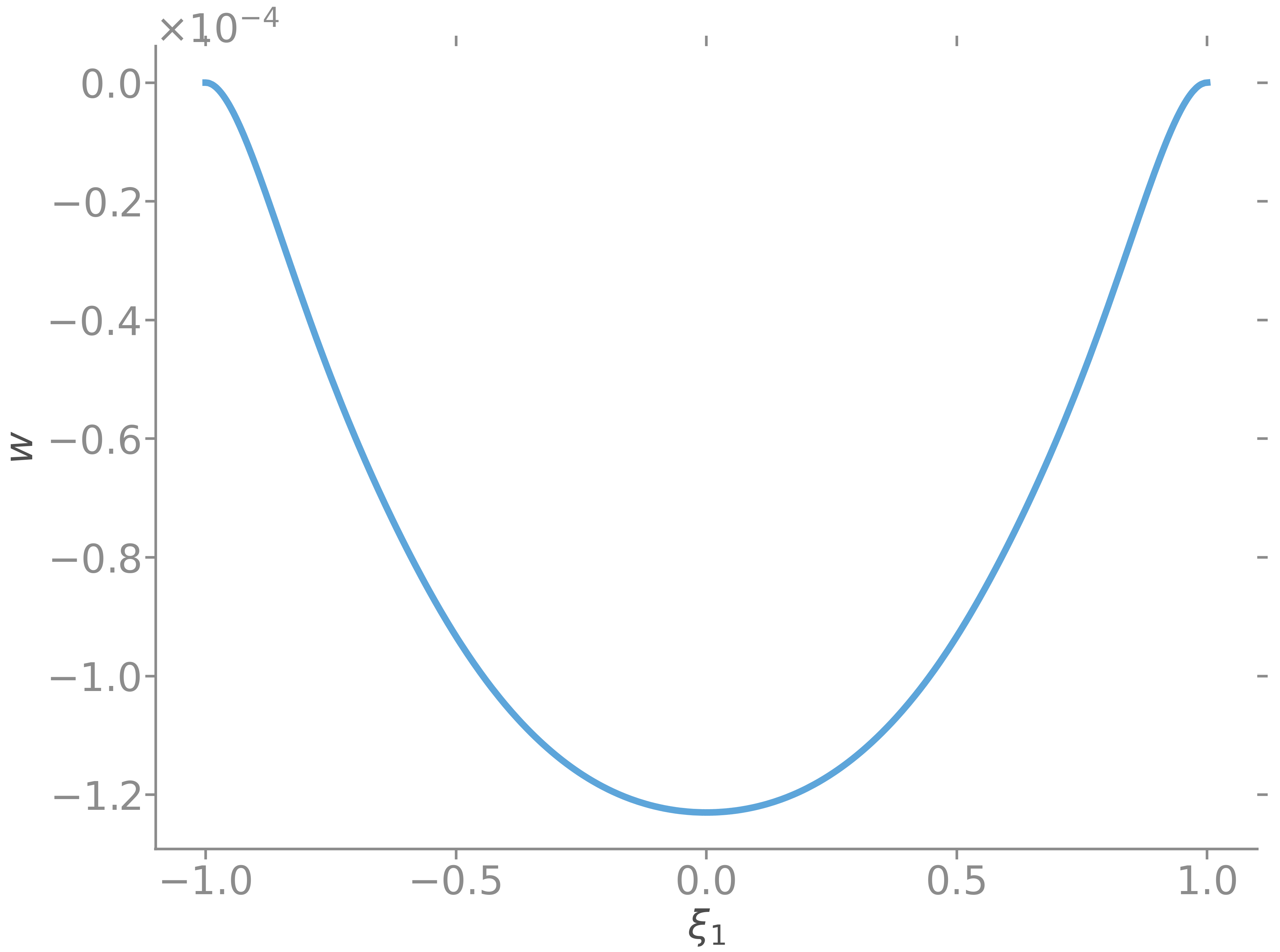}
\end{minipage}
\caption{LEFT: $v$. RIGHT: $w$, configuration of four negative disclinations with $s=-\tfrac{1}{2}$. Profiles evaluated for $\xi_1 \in [-1,1] $ and $\xi_2 = 0$.}
\label{fig:202412072139}
\end{figure}
%
%
%
%
%
Then, taking $s = +\frac{1}{2}$ in Eq. \eqref{eq:202412091308} yields a radial compressive mechanical stress (see Figure \ref{fig:fourpositivedisc}, the plot of $\sigma_{rr}$ is negative in $\Omega$).
%
%
\begin{figure}[htbp]
\centering
\includegraphics[width=1\textwidth]{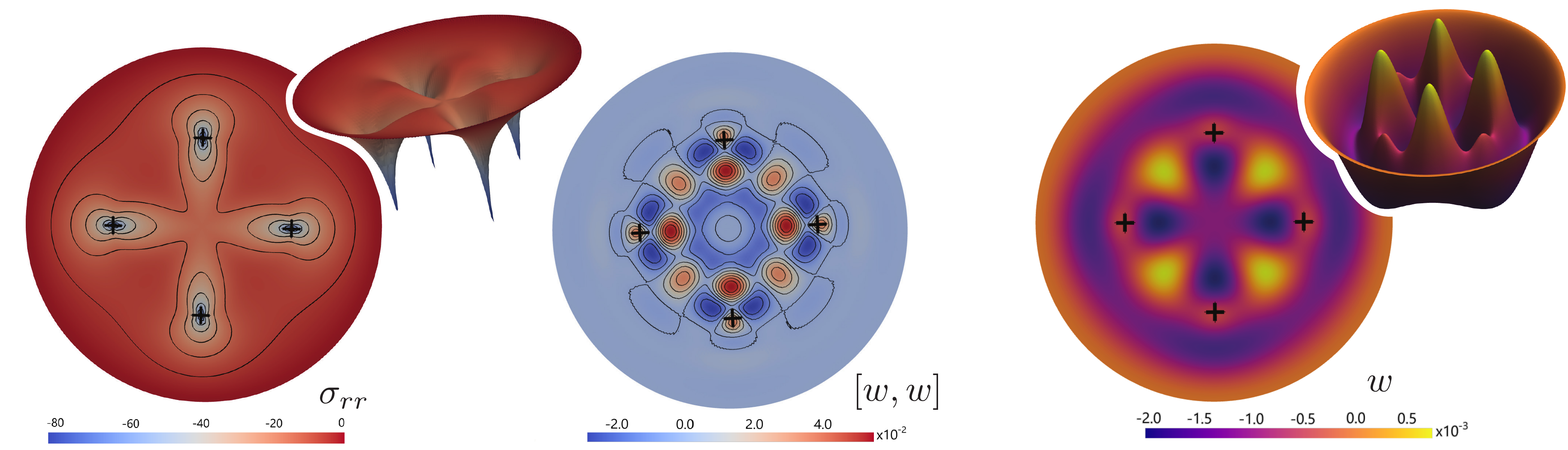} 
\caption{Configuration with four positive disclinations (black crosses) with $s=-\tfrac{1}{2}$. LEFT: radial stress $\sigma_{rr}$. Top and side views. CENTER: level curves and heat map of $[w,w]$. RIGHT: transverse displacement. Top and side views.}
\label{fig:fourpositivedisc}
\end{figure}
We report profiles of Airy functions and deflection along slices in Figures $\ref{fig:202412072139}$-LEFT and $\ref{fig:202412072139}$-RIGHT, respectively, for the case $s=-\frac{1}{2}$
and
in Fig. \ref{fig:202412072147}
for the case $s=\frac{1}{2}$.
In the case of negative disclinations, the Airy potential is negative. 
The behavior of the deflection is  complex.
For configurations with negative disclinations, we observe that the plate bends downward. For positive disclinations, although the deflection
$w$ is also negative, the compressive stress induced by the positive disclination distribution amplifies the bending of the plate. This effect leads to the formation of a fine folding pattern on the elastic surface, a phenomenon typical of clamped plates subjected to compressive membrane loads.
It is important to note that in this experiment, the plate is free from any external compressive loads; the stress state is solely induced by the presence of positive disclinations.

\textcolor{black}{An investigation into wrinkling and the associated mechanical tension states, as well as their relationship to material and geometric parameters and the resulting length scales, goes beyond the scope of our paper.
We refer to  \cite{DeS00},
\cite{DeS02}, where a detailed 3D-to-2D variational analysis has been conducted for thin structures to characterize and provide quantitative estimates of energy scaling under various regimes and wrinkling patterns
and \cite{gioia94} for investigation of folding patterns in FvK plates. 
For a review of wrinkling in graphene, we refer to \cite{deng16}.}

\begin{figure}[htbp]
\centering
\begin{minipage}{0.4\textwidth}
\centering
\includegraphics[width=\textwidth]{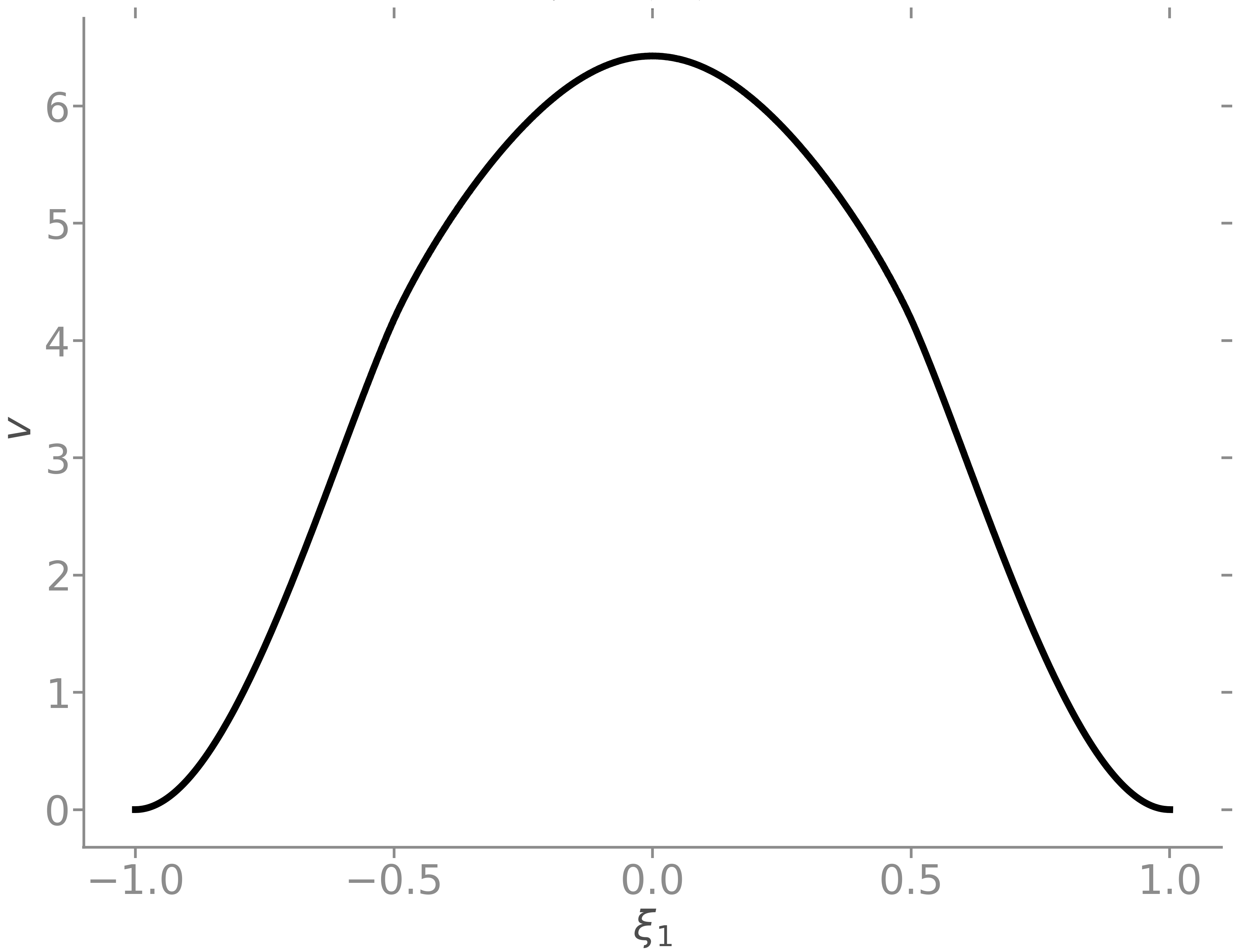} 
\end{minipage}
\hfill
\begin{minipage}{0.4\textwidth}
\centering
\includegraphics[width=\textwidth]{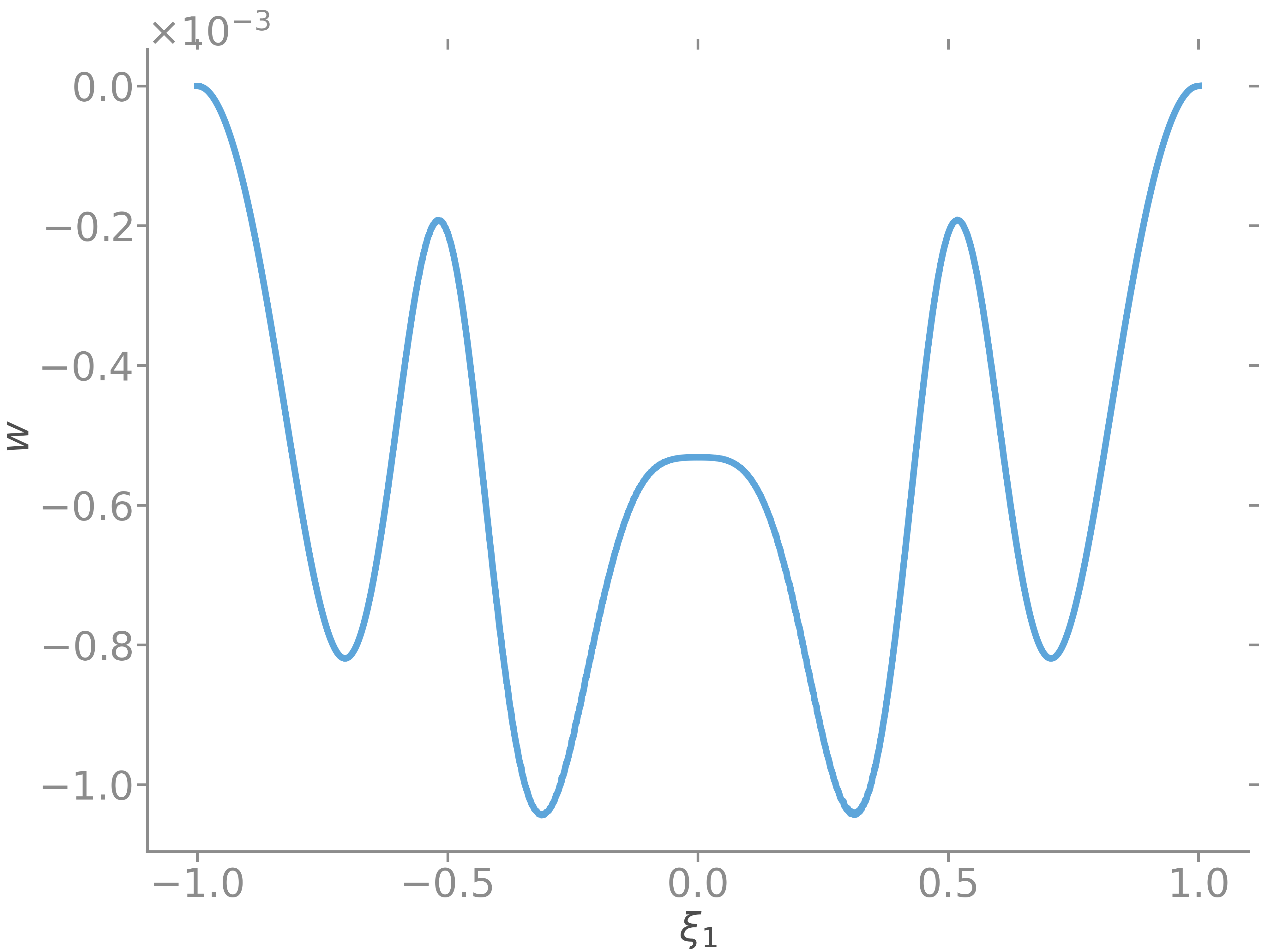}
\end{minipage}
\caption{LEFT: $v$. RIGHT: $w$, configuration of four positive disclinations with $s=\tfrac{1}{2}$. Profiles evaluated for $\xi_1 \in [-1,1] $ and $\xi_2 = 0$.}
\label{fig:202412072147}
\end{figure}
Recall that the quantity $[w, w]$ is twice the Gaussian curvature.
As shown in Figure \ref{fig:fournegativedisc}-RIGHT for $s=-\frac{1}{2}$ and in Fig. \ref{fig:fourpositivedisc}-CENTER for $s=\frac{1}{2}$, both cases exhibit a complex, non-trivial pattern, with regions of both positive and negative Gaussian curvature. 
This is a consequence of the specific choice  of boundary conditions.
In contrast, for a free-standing plate, 
a positive disclination is known to induce deformation with positive Gaussian curvature, while a negative disclination leads to negative Gaussian curvature.
In the present situation, we have homogeneous Dirichlet boundary conditions for the deflection 
$w$, resulting in the condition
\begin{equation}
\int_{\Omega} [w, w] \, d\xi = 0. 
\end{equation}
(To verify this, suffices to consider Eq. \ref{eq:202412131450} with $\phi=\eta=w$ and $\chi=1$).
This implies that the mean Gaussian curvature is zero.


In Figure \ref{fig:fourpositivedisc}-RIGHT we report the color map and the 3D plot of the plate mid-plane.
These benchmarks demonstrate that arrangements of disclinations angles of the same sign are capable of non-trivially distorting the plate and inducing regions of both positive and negative curvature coexisting within the plate.

\subsubsection{Zero total Frank angle}\label{}

We illustrate more realistic configurations with symmetric arrangements of the disclinations where the total Frank angle is zero. Similar flower-like patterns are observed in graphene \cite{Lu13}. Take
\begin{equation}
\label{eq:202501161551}
\theta(\xi) = \delta(\xi) -\frac{1}{4} \delta(\xi-y^{(1)}) - \frac{1}{4} \delta(\xi-y^{(2)}) - \frac{1}{4} \delta(\xi-y^{(3)}) - \frac{1}{4} \delta(\xi-y^{(4)})
\end{equation}
where  $y^{(1)} = (\frac{1}{2}, 0)$, $y^{(2)} = (0, \frac{1}{2})$, $y^{(3)} = (-\frac{1}{2}, 0)$, $y^{(4)} = (0, -\frac{1}{2})$. Solutions for defection and Gaussian curvature are displayed in Figure \ref{fig:202501171258}, showing more complex patterns of alternating regions of positive and negative Gaussian curvature.
\begin{figure}[htbp]
\centering
\includegraphics[width=.8\textwidth]{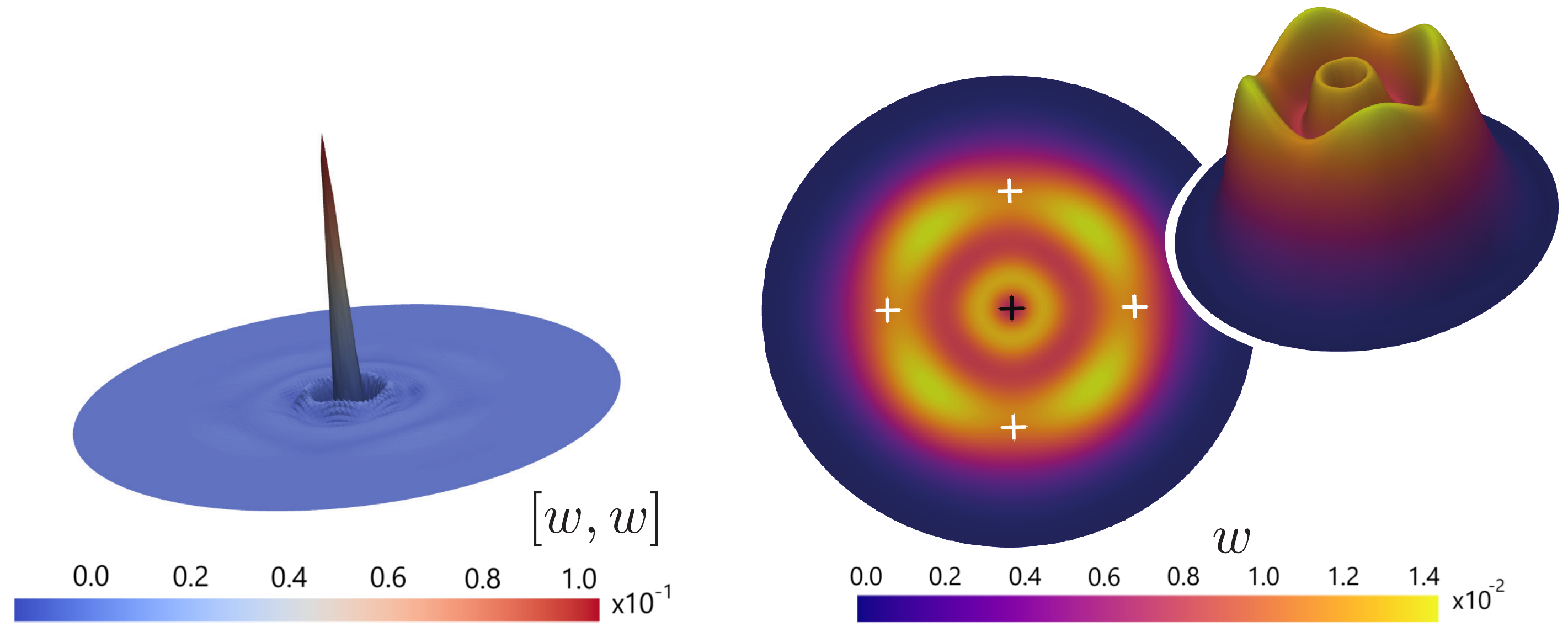} 
\caption{
LEFT: heatmap of $[w,w]$. 
The Gaussian curvature remains continuous and non-singular, even at the origin of the disk. The regularity Theorem \ref{Thm:202420121439} ensures $w\in C^{2,\zeta}$ and hence $[w,w]$ is of class $ C^{0,\zeta}$.
RIGHT: heatmap plot of $w$. 
The black (white) cross show the location of positive (negative) disclinations. 
}
\label{fig:202501171258}
\end{figure}
%
%
%
More solutions, corresponding to reversing the sign in Eq.\eqref{eq:202501161551} 
yielding 
%
\begin{equation}
\label{eq:202501171308}
\theta(\xi) = - \delta(\xi) +\frac{1}{4} \delta(\xi-y^{(1)}) + \frac{1}{4} \delta(\xi-y^{(2)}) + \frac{1}{4} \delta(\xi-y^{(3)}) + \frac{1}{4} \delta(\xi-y^{(4)}),
\end{equation}
are displayed in Figure \ref{fig:side_by_side}.


\begin{figure}[htbp]
\centering
\includegraphics[width=.8\textwidth]{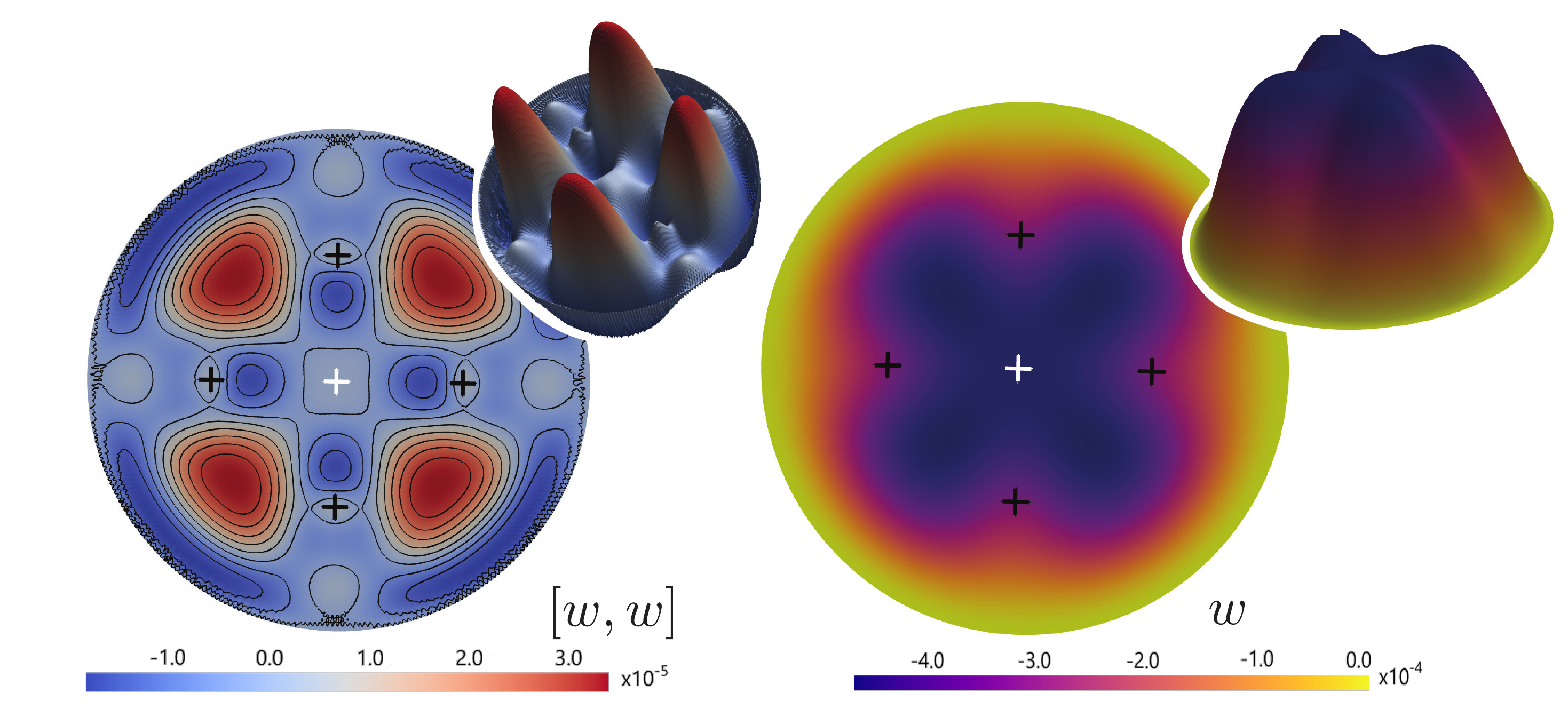} 
\caption{
LEFT: heatmap of $[w,w]$. 
Top and side views. 
RIGHT: heatmap plot of $w$. 
The black (white) cross show the location of positive (negative) disclination.}
\label{fig:side_by_side}
\end{figure}
%
%

\subsection*{Conclusion}



We have presented a variational framework for modeling thin plates with kinematic rotational incompatibility, based on the F\"oppl-von K\'arm\'an theory. A key component of this framework is an existence theorem, adapted from the classical theory by Ciarlet, which is   tailored for situations involving disclinations.   A regularity theorem, also developed following the  approach of Ciarlet, is presented. The theorem ensures the regularity of solutions, even in the presence of folding.
Building on this, we   present a variational formulation of the problem, which is numerically implemented using the Finite Element method. A Discontinuous Galerkin method with
$C^0$ elements is used to solve the F\"oppl-von K\'arm\'an equations.
Our implementation is developed within the FEniCS environment and leverages the variational formulation of the Dirichlet-type boundary condition problem for clamped FvK plates.


To evaluate the predictions based on this model,
 we test our code in various configurations over a wide range of parameters.
We present simulations of plates largely distorted by multi-disclination configurations, inspired by the literature on graphene.
Our simulations show that the presence of same-sign disclinations can lead to deformed plates with complex shapes, wrinkles, and including regions of both positive and negative curvature. More general configurations could be simulated, such as disclinations of opposite signs, dipoles, multipoles, as well as topological defects such as dislocations. These and other more complex configurations will be explored in future studies.

\paragraph{Acknowledgements}
EF is supported by JST SPRING, Grant Number JPMJSP2136.
PC's work is supported by JSPS KAKENHI
Grant-in-Aid for Scientific Research (C) JP24K06797.
PC holds an honorary appointment at La Trobe University and is a member of GNAMPA.
AALB   thanks   the Institute of Mathematics for Industry, an International Joint Usage and Research Center located in Kyushu University, where part of the work contained in this paper was carried out.

\paragraph{Supplementary material.}
Supplementary material for this article is available at \url{https://github.com/kumiori/disclinations}.

\appendix
\numberwithin{equation}{section}
\section{Appendix}
\subsection{Existence of solutions}

We break down the proof of the existence result in a series of propositions, following the lines of the proof given in \cite{ciarlet97}.
The following two propositions are an adaptation of \cite[Theorems 5.8-1, 5.8-2, 5.8-3]{ciarlet97}, incorporating the new operators $F_{\theta}, \Lambda_{\theta}$   to account for Dirac deltas in the Airy potential equation.

\begin{proposition}
	
	\label{prop:202412131015}
	Consider the assumptions of Theorem
	\ref{2408121907}. Then the pair $(\vexact, \wexact) \in H^2_0(\Omega) \times H^2_0(\Omega)$ satisfies \eqref{2406171748} if and only if it satisfies
	\begin{equation}
		\label{eq:202413121014}
		\frac{Eh}{2} C(w) + \big( D I - Eh \Lambda_{\theta} \big) w - F_{p} = 0 \quad \text{and} \quad v = - \frac{Eh}{2} B(w, w) + E h F_{\theta}
		\quad \text{ in } \Omega,
	\end{equation}
	where $I : H^2_0(\Omega) \to H^2_0(\Omega)$ is the identity operator and
	\begin{align*}
		& F_p \in H_0^2(\Omega)                                     &  & \text{such that} \quad \Delta^2 F_p = p \quad \text{in } H^{-2}(\Omega),                        \\
		& F_{\theta} \in H_0^2(\Omega)                              &  & \text{such that} \quad \Delta^2 F_{\theta} = \theta \quad \text{in } H^{-2}(\Omega),            \\
		& B \colon H^2(\Omega) \times H^2(\Omega) \to H^2_0(\Omega) &  & \text{such that} \quad \Delta^2 B(\zeta, \eta) = [\zeta, \eta] \quad \text{in } H^{-2}(\Omega), \\
		& C \colon H_0^2(\Omega) \to H^2_0(\Omega)                  &  & \text{such that} \quad \eta \mapsto B(B(\eta, \eta), \eta),                                     \\
		& \Lambda_{\theta} \colon H_0^2(\Omega) \to H^2_0(\Omega)   &  & \text{such that} \quad \eta \mapsto B(F_{\theta}, \eta).
	\end{align*}
\end{proposition}
\begin{proof}
	($\Rightarrow$) We now assume $(\vexact, \wexact)$ solves Eq. \eqref{2406171748} and show that the pair also satisfies Eq. \eqref{eq:202413121014}. By definition of $B$, $F_\theta$ and by linearity of bilaplacian operator, Eq. \eqref{2406171748}-(2) can be reformulated in terms of the following Dirichlet problem
	\begin{equation}
		\label{eq:202412131034}
		\begin{cases}
			\Delta^2 ( \frac{1}{Eh} \vexact + \frac{1}{2} B(\wexact, \wexact) -
			F_{\theta}) = 0 \text{ in } H^{-2}(\Omega) \\
			\big( \frac{1}{Eh} v + \frac{1}{2}B(\wexact, \wexact) - F_{\theta} \big) \in H^2_0(\Omega),
		\end{cases}
	\end{equation}
	whose unique solution is the trivial solution; this proves the identity \eqref{eq:202413121014}-(2). Analogously, Eq. \eqref{2406171748}-(1) can be written as
	\begin{equation}
		\label{eq:202412131037}
		\begin{cases}
			\Delta^2 ( D \wexact - B(\vexact, \wexact) -
			F_{p}) = 0 \text{ in } H^{-2}(\Omega) \\
			\big( D \wexact - B(\vexact, \wexact) -
			F_{p} \big) \in H^2_0(\Omega),
		\end{cases}
	\end{equation}
	whence $D \wexact = B(\vexact, \wexact) + F_{p}$ in $\Omega$. Plugging identity \eqref{eq:202413121014} into it and exploiting the bilinearity of $B$ (which holds in view of the bilinearity of the Monge-Amp\`ere form) one obtains
	\begin{align}
		D \wexact = - \frac{Eh}{2} B(B(\wexact, \wexact), \wexact) + Eh B(F_\theta, \wexact) +  F_{p}.
	\end{align}
	The claim follows by the definitions of the operators $C$ and $\Lambda_{\theta}$. \\
	($\Leftarrow$) Conversely, we assume now $(\vexact, \wexact)$ solves Eq. \eqref{eq:202413121014} and show that the pair also satisfies Eq. \eqref{2406171748}. Combining Eq. \eqref{eq:202413121014}-(1) and (2) we obtain $D \wexact = B(\vexact, \wexact) +  F_{p}$, taking now the bilaplacian on both sides one   obtains Eq. \eqref{2406171748}-(1). Analogously, taking the bilaplacian operator on both sides of \eqref{eq:202413121014}-(2) one readily obtain \eqref{2406171748}-(2). 
\end{proof}
\begin{proposition}[properties of the operators $B$, $C$ and $\Lambda_{\theta}$]

	\label{prop:202412131230}
	Let $\chi$ in $H^2(\Omega)$, $\phi$, $\eta$ $ \in H_0^2(\Omega)$, $(\phi_k)_{k \in \mathbb{N}} \in H_0^2(\Omega)$ and assume $\phi_k \rightharpoonup \phi $ weakly in $H^2(\Omega)$. Then the following properties hold
	\begin{align}
		\label{eq:1b}
		(i) \quad   & (B(\chi, \phi), \eta)_{\Delta} = (B(\chi, \eta), \phi)_{\Delta}, \notag                                                                                                            \\
		(ii) \quad  & (C(\phi), \phi)_{\Delta} = (B(\phi, \phi), B(\phi, \phi))_{\Delta} \ge 0,  \notag                                                                                                  \\
		(iii) \quad & (C(\phi), \phi)_{\Delta} = 0 \iff \phi \equiv 0,                                                                                                                                   \\
		(iv) \quad  & (\Lambda_{\theta}\phi, \eta)_{\Delta} = (\Lambda_{\theta}\eta, \phi)_{\Delta}, \notag                                                                                              \\
		(v) \quad   & C(\phi_k) \rightarrow C(\phi) \text{ and } \Lambda_{\theta} \phi_k \rightarrow \Lambda_{\theta} \phi \text{ strongly in } H^2(\Omega) \notag \quad \text{(sequential compactness).}
	\end{align}

\end{proposition}
\begin{proof}
	A proof of (i), (ii), (iii) can be found in \cite{ciarlet97}, Thm. 5.8-2. (iv) can be obtained from (i), (v) is a direct consequence of the sequential compactness of the operator $B$ (see \cite{ciarlet97}, Thm. 5.8-2). 
\end{proof}

Finally, we are in a position to  prove the existence of solutions to FvK equations in the presence of wedge disclinations.

\begin{proof}[Proof of Theorem \ref{2408121907} (Existence)]\label{thm:202412201440}
	By Proposition \ref{prop:202412131015} we only need to show that there exists some $\wexact \in H^2_0(\Omega)$ solution to Eq. \eqref{eq:202413121014}-(1). We proceed via Direct Method of Calculus of Variations by showing that the functional $\mathcal{J} : H^2_0(\Omega) \rightarrow \mathbb{R} $, defined as 
	\begin{equation}
		\label{eq202412131200}
		\mathcal{J}(w) := \frac{Eh}{8} \big( C(w), \eta \big)_{\Delta} + \frac{D}{2} \big( w , w \big)_{\Delta} - \frac{Eh}{2} \left( \Lambda_{\theta} w , w \right)_{\Delta} - \big( F_{p}, w \big)_{\Delta}
	\end{equation}
	whose Euler-Lagrange equation is precisely Eq. \eqref{eq:202413121014}-(1), attains its minimum in $H^2_0(\Omega)$. Because $(\cdot, \cdot)_{\Delta}$ and $B(\cdot, \cdot)$ are continuous bilinear operators, they are Fréchet differentiable and, as a consequence, so is $\mathcal{J}$. Simple calculations show that the Euler-Lagrange equation of $\mathcal{J}$ is  indeed Eq. \eqref{eq:202413121014}-(1).
	To show the coercivity of $\mathcal{J}$ in $H^2_0(\Omega)$ we follow \cite{ciarlet97} and argue by contradiction. Consider a sequence $(w_k)_{k \in \mathbb{N}}$  and assume there exists $M_1 > 0$ such that $\lim_{k \to \infty} \mathcal{J}(w_k) < M_1 $ when $\lVert w_k \rVert_{H^2(\Omega)} \to +\infty$. Define the rescaled sequence $(\zeta_k)_{k\in \mathbb{N}}$ as $\zeta_k := w_k / \lVert \Delta w_k \rVert_{L^2(\Omega)}$. Then $(\zeta_k)_{k\in \mathbb{N}}$ is bounded in $H^2(\Omega)$-norm and, up to some subsequence $(\zeta_{k_\ell})_{\ell \in \mathbb{N}}$, it converges weakly to some $\zeta_{\infty} \in$  $H_0^2(\Omega)$. For $\ell$ big enough we can assume $\lVert \Delta w_{k_{\ell}} \rVert_{L^2(\Omega)} \ne 0 $, then $\frac{\mathcal{J}(w_{k_{\ell}})}{\lVert \Delta w_{k_{\ell}} \rVert^{2}_{L^2(\Omega)}} \le \frac{M_1}{\lVert \Delta w_{k_{\ell}} \rVert^{2}_{L^2(\Omega)}}$ reads as
	\begin{equation}
		\label{eq:202412261034}
		\frac{Eh}{8} \lVert \Delta w_{k_{\ell}} \rVert^2_{L^2(\Omega)} \left( C(\zeta_{k_{\ell}}), \zeta_{k_{\ell}} \right)_{\Delta} + \frac{D}{2} - Eh \left( \Lambda_{\theta} \zeta_{k_{\ell}} , \zeta_{k_{\ell}} \right)_{\Delta} - \frac{\left( F_{p}, \zeta_{k_{\ell}} \right)_{\Delta}}{\lVert \Delta w_{k_{\ell}} \rVert_{L^2(\Omega)}} \le \frac{M_1}{{\lVert \Delta w_{k_{\ell}} \rVert^{2}_{L^2(\Omega)}}}
	\end{equation}
	where we used that for every $\alpha \in \mathbb{R}$ it holds: $C(\alpha w) = \alpha^3 C(w)$ and $\Lambda_{\theta} \alpha w $ $= \alpha \Lambda_{\theta} w $. By Proposition \ref{prop:202412131230}-(ii) $\left( C(\zeta_{k_{\ell}}), \zeta_{k_{\ell}} \right)_{\Delta}$ is always non-negative. Take $\liminf_{\ell \to \infty} (\cdot)$ on both sides of \eqref{eq:202412261034} and suppose first $\liminf_{\ell \to \infty} \left( \lVert \Delta w_{k_{\ell}} \rVert^2_{L^2(\Omega)} \left( C(\zeta_{k_{\ell}}), \zeta_{k_{\ell}} \right)_{\Delta} \right)$ diverges.
	This leads to a contradiction since the remaining terms of the LHS of \eqref{eq:202412261034} converge to some real number, and the RHS of \eqref{eq:202412261034} converges to zero. Suppose hence
	\begin{equation}
		\liminf_{\ell \to \infty} \left( \lVert \Delta w_{k_{\ell}} \rVert^2_{L^2(\Omega)} \left( C(\zeta_{k_{\ell}}), \zeta_{k_{\ell}} \right)_{\Delta} \right) \le M_2
	\end{equation}
	for some $M_2 \ge 0$, then $ \left( C(\zeta_{k_{\ell}}), \zeta_{k_{\ell}} \right)_{\Delta}$ converges (by virtue of Lemma \eqref{prop:202412131230}-(v)) to zero; by Lemma \ref{prop:202412131230}-(c) $\zeta_{\infty} \equiv 0$, which in turn implies that $\left( \Lambda_{\theta} \zeta_{k_{\ell}} , \zeta_{k_{\ell}} \right)_{\Delta} \to 0$ when $\ell \to \infty$. Hence
	\begin{equation}
		\liminf_{\ell \to \infty} \left( \lVert \Delta w_{k_{\ell}} \rVert^2_{L^2(\Omega)} \left( C(\zeta_{k_{\ell}}), \zeta_{k_{\ell}} \right)_{\Delta} \right) + \frac{D}{2} \le 0
	\end{equation}
	which shows the contradiction. The weak lower semi-continuity of $\mathcal{J}$ is guaranteed by Proposition \ref{prop:202412131230}-(v). By Direct Method of the Calculus of Variations, $\mathcal{J}(w)$ attains its minimum in $H^2_0(\Omega)$. 
\end{proof}

\subsection{Regularity of solutions}

The proof we present follows the lines of \cite[Theorem  5.8-4]{ciarlet97}, which  we extend to the case at hand. \\\\
\begin{proof}[Theorem \ref{Thm:202420121439} (Regularity)]
	We start by studying the regularity of $w$. By assumption $v$, $w$ belong to $H_0^2(\Omega)$. Then $[w,v] \in L^1(\Omega)$ and, as a consequence $\Delta^2 w$ $\in L^1(\Omega)$. For $\epsilon \in (0,1)$ it holds $L^1(\Omega) \hookrightarrow H^{-1-\epsilon}(\Omega)$ where the symbol $\hookrightarrow$ denotes the continuous embedding (see \cite{ciarlet97}, Thm. 4.6-3, page 291) and $H^{-1-\epsilon}(\Omega)$$:= (H_0^{1+\epsilon}(\Omega))^{'} $, $H^{1+\epsilon}_0(\Omega) := \overline{C^{\infty}_0(\Omega)}^{\lVert \cdot \rVert_{H^{1+\epsilon}}} $ (see for example \cite{HitchhikerGuide}). Then $\Delta^2 w \in H^{-1-\epsilon}(\Omega)$ and \color{black} for $\overline{\epsilon} \in (1/2, 1)$ (see \cite[Thm. 6]{savare}) \color{black} $w \in H^{3-\overline{\epsilon}}(\Omega) \cap H_0^2(\Omega)$.
	Then $\nabla^2 w \in H^{1-\overline{\epsilon}}(\Omega, \mathbb{R}^{2\times2})$, and from the theory of continuous embedding in fractional Sobolev spaces (see \cite[Thm. 6.7]{HitchhikerGuide}), we have $H^{1-\overline{\epsilon}}(\Omega, \mathbb{R}^{2\times2}) \hookrightarrow L^{2/ \overline{\epsilon}}(\Omega, \mathbb{R}^{2\times2})$, which implies, by H\"older inequality, that $[v,w] \in L^{2 /(1+ \overline{\epsilon} )}(\Omega)$. Then $\Delta^2 w \in L^{2 /(1+ \overline{\epsilon} )}(\Omega)$, and by the theory of $L^p$-regularity for biharmonic problems, $w \in W^{4, 2 / ( 1+ \overline{\epsilon} )}(\Omega)$ (see \cite[Thm. 2.20]{Gazzola2010PolyharmonicBV}). By standard Sobolev continuous embeddings $w \in W^{4, 2 / ( 1+ \overline{\epsilon} ) }(\Omega) \hookrightarrow C^{2, 1-\overline{\epsilon}}(\overline{\Omega})$ (\cite[Thm. 2.6]{Gazzola2010PolyharmonicBV}), hence $\nabla^2 w \in C^{0, 1-\overline{\epsilon}}(\overline{\Omega}, \mathbb{R}^{2\times2})$ and by H\"older inequality $[v, w] \in L^2(\Omega)$, then $\Delta^2 w \in L^2(\Omega)$. In light of \cite[Thm. 2.20]{Gazzola2010PolyharmonicBV} we have $w \in H^4(\Omega)$. \\
	We now turn  our attention to the regularity of $v$. Define the function $v_\text{p}: \overline{\Omega} \to \mathbb{R}$ as
	\begin{equation}
		v_\text{p}(x) := \sum_{i = 1}^N s_i \mathcal{G}(x;y^{(i)}) \quad \text{where} \quad
		\mathcal{G}(x;y^{(i)}) :=
		\begin{cases}
			\sum\limits_{i = 1}^{N} \frac{1}{16\pi} \left|x-y^{(i)}\right|^2 \displaystyle \ln\left(|x-y^{(i)}|^2\right) \quad &\text{for every } x \ne y^{(i)} \\
			0 \quad &\text{otherwise }
		\end{cases}
	\end{equation}
	It can be verified that $\Delta^2 v_{\text{p}} = \theta $ and, owing to \cite[Example 4.25]{Demengel2012}, $v_{\text{p}} \in W^{2,\tau}(\Omega) \cap C^{\infty}(\Omega \setminus \cup_{i=1}^{N} B_{r}(y^{(i)}))$ for every $\tau \in [1,\infty)$ and every $r > 0$ such that $\overline{B_{r}(y^{(i)})} \subset \Omega$ for every $i \in \{1, \hdots, N\}$. 
	Define now $v_{\text{e}} := v - v_{\text{p}}$. Then, $v_{\text{e}}$ satisfies the following Dirichlet problem 
	\begin{equation}
		\label{eq:202412111549}
		\begin{cases}
			\frac{1}{Eh} \Delta^2 v_{\text{e}} = -\frac{1}{2}[w,w] \quad &\text{ in } \Omega \\
			v_{\text{e}} = - v_{\text{p}} \quad &\text{ on } \partial \Omega \\
			\partial_n v_{\text{e}} = - \partial_n v_{\text{p}} \quad &\text{ on } \partial \Omega
		\end{cases}
	\end{equation}
	where, 
	in view of the Sobolev embeddings $H^4(\Omega) \hookrightarrow C^{2, s}(\overline{\Omega})$ for every
	$s \in (0,1)$, we have $w \in C^{2,s}(\overline{\Omega})$ and   $[w,w] \in C^{0,s}(\overline{\Omega})$. By Schauder regularity (see \cite[section 2.5.1, Thm. 2.19]{Gazzola2010PolyharmonicBV}) we conclude that $v_{\text{e}} \in C^{4,\gamma}(\overline{\Omega})$ 
	\color{black} and, in turn that $v \in W^{2,\tau}(\Omega) \cap C^{4,\gamma}(\Omega \setminus \cup_{i=1}^{N} B_{r}(y^{(i)})) $. \\
	We may reconsider now the regularity of $w$. By H\"older inequality $[v, w] \in L^{\tau }(\Omega)$ and therefore, by \cite[Thm. 2.6]{Gazzola2010PolyharmonicBV}, $w \in W^{4,k}(\Omega)$.
\end{proof}

\subsection{Property of the Monge-Amp\`ere bracket}

\begin{lemma}
	\label{lemma:202412131500}
	Let $\Omega \subset \mathbb{R}^2$ as in Theorem \ref{2408121907}, $\phi \in H^2_0(\Omega)$ and $\chi, \eta \in H^2(\Omega)$. Then the following identities hold
	\begin{equation}
		\label{eq:202412131450}
		\int_{\Omega} \cof(\nabla^2 \phi) : (\nabla \chi \otimes \nabla \eta) \, dx =
		- \int_{\Omega} [\chi,\eta] \, \phi \, dx = - \int_{\Omega} [\phi,\eta] \, \chi \, dx.
	\end{equation}
\end{lemma}

\begin{proof}
	Let $(\phi_k)_{k \in \mathbb{N}} \in C^{\infty}_0(\Omega)$, $(\chi_\ell)_{\ell \in \mathbb{N}}$, $(\eta_m)_{m \in \mathbb{N}}$ $\in C^{\infty}(\overline{\Omega})$. Then
	\begin{equation*}
		\int_{\Omega} \cof(\nabla^2 \phi_k) : (\nabla \chi_\ell \otimes \nabla \eta_m) \, dx = \int_{\Omega} \cof(\nabla^2 \phi_k) \nabla \chi_\ell \cdot \nabla \eta_m \, dx =
	\end{equation*}
	\begin{equation*}
		= \int_{\partial \Omega} \eta_m \, \cof(\nabla^2 \phi_k) \nabla \chi_\ell \cdot n \, d\mathcal{H}^1 - \int_{\Omega} \eta_m \, \text{div} \left( \cof(\nabla^2 \phi_k) \nabla \chi_\ell \right) \, dx =
	\end{equation*}
	\begin{equation*}
		= \int_{\partial \Omega} \eta_m \, \cof(\nabla^2 \phi_k) n \cdot \nabla \xi_\ell \, d\mathcal{H}^1 - \int_{\Omega} \eta_m \, \text{div} \left( \cof(\nabla^2 \phi_k) \right) \cdot \nabla \chi_\ell \, dx - \int_{\Omega} \eta_m \, \cof(\nabla^2 \phi_k) : \nabla^2 \chi_\ell \, dx.
	\end{equation*}
	Since $\varphi_k \in C^{\infty}_0(\Omega)$ the boundary integral vanishes; simple calculations show that $\text{div} \left( \cof(\nabla^2 \phi_k) \right) = 0$. Then
	\begin{equation}
		\label{eq:202412131430}
		\int_{\Omega} \cof(\nabla^2 \phi_k) : (\nabla \chi_{\ell} \otimes \nabla \eta_{m}) \, dx = - \int_{\Omega} [\phi_k, \chi_{\ell}] \eta_{m} \, dx.
	\end{equation}
	We proceed now by density of $C^{\infty}(\overline{\Omega})$ and $C_0^{\infty}(\Omega)$ in $H^2(\Omega)$ and $H^2_0(\Omega)$, respectively, (see \cite[Theorem 3.18]{Adams1975}).
	Take the limit of $\chi_\ell \to \chi$ strongly in $H^2(\Omega)$ as $\ell\to\infty$ yielding
	\begin{equation}
		\label{eq:202412251230}
		\int_{\Omega} \cof(\nabla^2 \phi_k) : (\nabla \chi \otimes \nabla \eta_{m}) \, dx = - \int_{\Omega} [\phi_k, \chi] \eta_{m} \, dx.
	\end{equation}
	Then, take the limit $\eta_m \to \eta$ strongly in $H^2(\Omega)$ as $m\to\infty$ to obtain
	\begin{equation}
		\label{eq:202412251240}
		\int_{\Omega} \cof(\nabla^2 \phi_k) : (\nabla \chi \otimes \nabla \eta) \, dx = - \int_{\Omega} [\phi_k, \chi] \eta \, dx.
	\end{equation}
	Finally, let $\phi_k \to \phi$ strongly in $H^2(\Omega)$. In view of the continuous embedding $H^2(\Omega) \hookrightarrow W^{1,4}(\Omega)$, the term $(\nabla \chi \otimes \nabla \eta) \in L^2(\Omega, \mathbb{R}^{2\times2})$ and we obtain
	\begin{equation}
		\label{eq:202412251242}
		\int_{\Omega} \cof(\nabla^2 \phi) : (\nabla \chi \otimes \nabla \eta) \, dx = - \int_{\Omega} [\phi, \chi] \eta \, dx.
	\end{equation}
	To conclude the proof, we use the symmetry of $\int_{\Omega} [\phi, \chi] \,\eta \, dx$, that is,  $\int_{\Omega} [\phi, \chi] \,\eta \, dx = $ $\int_{\Omega} [\phi, \eta] \, \chi \, dx = $ $\int_{\Omega} [\chi, \eta] \, \phi \, dx$, see \cite[Theorem 5.8-2]{ciarlet97}.

\end{proof}

\bibliographystyle{unsrt}
\bibliography{MAIN.bib}

\end{document}